\documentclass[11pt]{article}
\usepackage{graphicx}
\usepackage{amsthm, amsmath, amssymb, tikz, bm}
\usepackage{mathtools}
\usepackage{xifthen}
\usepackage{caption}
\usepackage[colorlinks=true,
linkcolor=blue,citecolor=blue,
urlcolor=blue]{hyperref}
\usepackage[margin=0.8in]{geometry} 

\usepackage[shortlabels]{enumitem}
\usepackage{todonotes}
\usetikzlibrary{calc,shapes, backgrounds}
\allowdisplaybreaks

\newtheorem{lemma}{Lemma}
\newtheorem{theorem}{Theorem}
\newtheorem{corollary}{Corollary}

\newtheorem{remark}{Remark}
\newtheorem{conjecture}{Conjecture}

\newcommand{\dss}{\displaystyle\sum}

\newcommand{\vol}{{\rm Vol}}

\newcommand{\WL}[1]{\textcolor{magenta}{{}Comment by WL: #1}}

\newcommand{\bx}{\mathbf{x}}
\newcommand{\by}{\mathbf{y}}
\newcommand{\bz}{\mathbf{z}}

\title{On the maximum second eigenvalue of outerplanar graphs}
\author{George Brooks
\thanks{University of South Carolina, Columbia, SC. ({\tt ghbrooks@email.sc.edu}). The author is partially supported by NSF DMS 2038080 grant.}
\and
Maggie Gu \thanks{University of North Carolina. ({\tt  meiqi.gu@gmail.com}).
 The author is partially supported by NSF DMS 2038080 grant through a summer REU program. }
 \and Jack Hyatt \thanks{University of South Carolina. ({\tt  jahyatt@email.sc.edu}).
 The author is partially supported by NSF DMS 2038080 grant through a summer REU program. }
\and 
William Linz \thanks{University of South Carolina, Columbia, SC. ({\tt wlinz@mailbox.sc.edu}). The author is partially supported by NSF DMS 2038080 grant.}
\and Linyuan Lu \thanks{University of South Carolina, Columbia, SC. ({\tt lu@math.sc.edu}). The author is partially supported by NSF DMS 2038080 grant.}
}

\begin{document}

\maketitle

\begin{abstract}
For a fixed positive integer $k$ and a graph $G$, let $\lambda_k(G)$ denote the $k$-th largest eigenvalue of the adjacency matrix of $G$. In 2017, Tait and Tobin~\cite{TT2017} proved that the maximum $\lambda_1(G)$ among all outerplanar graphs on $n$ vertices is achieved by the fan graph $K_1\vee P_{n-1}$. In this paper, we consider a similar problem of determining the maximum $\lambda_2$ among all connected outerplanar graphs on $n$ vertices. For $n$ even and sufficiently large, we prove that the maximum $\lambda_2$ is uniquely achieved by  the graph $(K_1\vee P_{n/2-1})\!\!-\!\!(K_1\vee P_{n/2-1})$, which is obtained by connecting two disjoint copies of $(K_1\vee P_{n/2-1})$ through a new edge
joining their smallest degree vertices. When $n$ is odd and sufficiently large, the extremal graphs are not unique. The extremal graphs are those graphs $G$ that contain a cut vertex $u$ such that $G\setminus \{u\}$ is isomorphic to $2(K_1\vee P_{n/2-1})$. 
We also determine the maximum $\lambda_2$ among all 2-connected outerplanar graphs and asymptotically determine the maximum of $\lambda_k(G)$ among all connected outerplanar graphs for any fixed $k$.
\end{abstract}

\section{Introduction}
Let $G = (V(G), E(G))$ be a graph on $n$ vertices, with $n$ a positive integer. The \textit{adjacency matrix} of $G$ is the $n\times n$ matrix $A = (a_{uv})_{u,v \in V(G)}$ with rows and columns indexed by $V(G)$ where $a_{uv} = 1$ if $u\sim v$ and $a_{uv} = 0$ otherwise. The \textit{eigenvalues} of $G$ are the eigenvalues of its adjacency matrix $A$. As $A$ is symmetric, these eigenvalues are real, and we label them in non-increasing order as $\lambda_1 \ge \lambda_2 \ge \cdots \ge  \lambda_n$.

The largest eigenvalue $\lambda_1$ is called the \textit{spectral radius} of a graph and has been extensively studied (see, for example, the older survey of Cvetkovi\'c and Rowlinson~\cite{Cvetkovic-Rowlinson90} or the more recent monograph by Stani\'c~\cite{Stanic2015}). More recently, there has been much interest in finding the maximum spectral radius over a family of graphs, as such results extend Tur\'an-type results via the inequality $\lambda_1 \ge \frac{2m}{n}$, where $m$ is the number of edges in the graph. An archetypal result is the spectral Tur\'an theorem of Nikiforov~\cite{Nik2, Nik0}, which determines the maximum spectral radius of a $K_{r+1}$-free graph. The maximum spectral radius has also been determined for many other classes of graphs, including outerplanar graphs and planar graphs~\cite{TT2017}, $K_{s, t}$-minor-free graphs~\cite{Nikiforov2017, Tait2019, ZL2022} and graphs forbidding a cycle or path of some prescribed length~\cite{Nik3, GH2019, YWZ2012, WZ2012}. 

\begin{figure}[hbt]
    \centering
\resizebox{5cm}{!}{\begin{tikzpicture}[scale=1, Wvertex/.style={circle, draw=black, fill=white, scale=1}, bvertex/.style={circle, draw=black, fill=black, scale=0.3}]

\node [bvertex] (v0) at (-0.5, -0.5) {};
\node [bvertex] (v1) at (-3, 1) {};
\node [bvertex] (v2) at (-2, 1) {};
\node [bvertex] (v3) at (-1, 1) {};
\node [bvertex] (v4) at (0, 1) {};
\node [bvertex] (v5) at (1, 1) {};
\node [bvertex] (v6) at (2, 1) {};

\draw (v0) -- (v1);
\draw (v0) -- (v2);
\draw (v0) -- (v3);
\draw (v0) -- (v4);
\draw (v0) -- (v5);
\draw (v0) -- (v6);

\draw (v1) -- (v2);
\draw (v3) -- (v2); 
\draw [dashed] (v4) -- (v3) node [midway, fill=white, above=3pt] {$...$};
\draw (v5) -- (v4);
\draw (v6) -- (v5);

\end{tikzpicture}}%
    \caption{The fan graph $K_1\vee P_{n-1}$.}
    \label{fig:fangraph}
\end{figure}
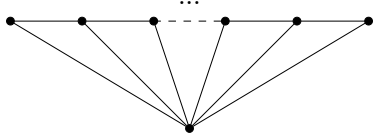

In contrast to the spectral radius, the $k$th largest eigenvalue of a graph has not been well-studied for $k\ge 3$. The second largest eigenvalue of a graph has been reasonably well-studied because of its relation to connectivity and expansion properties of graphs (see, for example, the survey of Cvetkovi\'c and Simi\'c~\cite{CS95} and the book of Stani\'c~\cite{Stanic2015}). The graphs with largest second eigenvalue have been determined for trees~\cite{Neumaier82, Shao95, Hof} and connected graphs~\cite{Hon, ZLW2012}. However, there are many classes of graphs for which the graph with maximum second largest eigenvalue is unknown, let alone the graph with maximum $k$th largest eigenvalue. Even the graph with maximum $k$th largest eigenvalue over the family of all graphs on $n$ vertices is unknown for any $k\ge 3$~\cite{Hon, Nik1, Linz2023}. The tree with maximum $k$th largest eigenvalue has been determined in a sequence of papers~\cite{Neumaier82, Yuan86, Shao95, Chen07}. 

In this paper, we study the second eigenvalue $\lambda_2$ of connected outerplanar graphs. A graph is \textit{outerplanar} if it can be drawn in the plane without crossing edges such that each of its vertices lies on the outer face. Equivalently, a graph is outerplanar if and only if it is $K_{2,3}$-minor-free and $K_4$-minor-free.  Cvetkovi\'c and Rowlinson~\cite{Cvetkovic-Rowlinson90} conjectured that the \textit{fan graph} $K_1 \vee P_{n-1}$ (see Figure~\ref{fig:fangraph}) is the outerplanar graph with maximum spectral radius. Work was done on this conjecture by Rowlinson~\cite{Rowlinson90} and Zhou, Lin and Hu~\cite{ZLH01}. The conjecture was proved for sufficiently large $n$ by Tait and Tobin~\cite{TT2017}. Subsequently, the outerplanar graph with maximum spectral radius was determined for all values of $n$ by Lin and Ning~\cite{Lin-Ning19}. 

By contrast, not much seems to be known about the middle eigenvalues of outerplanar graphs. Li and Sun~\cite{LS2023} have studied outerplanar graphs with small $\lambda_2$, but the $k$th largest eigenvalues of outerplanar graphs with $k\ge 2$ has not been previously studied.  

For a fixed $k$, we define $\lambda_{k,max}(n)$, or $\lambda_{k,max}$ for short, to be the maximum $k$th eigenvalue among all connected outerplanar graphs on $n$ vertices.  We have the following general theorem, which determines $\lambda_{k,max}(n)$ up to the second order.

\begin{theorem}\label{thm:t1}
For any fixed integer $k\geq 2$, we have
\begin{equation*}
\lambda_{k,max}(n)=\sqrt{n/k}+ 1+ O\left(\frac{1}{\sqrt{n}} \right).
\end{equation*}
Moreover, any outerplanar graph $G$ on $n$ vertices with $\lambda_k(G) =\lambda_{k,max}(n)$ contains exactly $k$ vertices of degree $\frac{n}{k} \pm O(1)$.
\end{theorem}

\begin{conjecture}\label{conjkq+1}
    Suppose $n=kq+1$. We conjecture that for a fixed $k$ and sufficiently large $n$,
    $$\lambda_{k,max}(n)=\lambda_1(K_1\vee P_{q-1}).$$
     Moreover, any extremal graph $G$ on $n$ vertices with $\lambda_k(G)=\lambda_{k,max}(n)$ has the following structure: there exist a cut vertex $u$ such that deleting $u$ from $G$ results in $k$ copies of $K_1\vee P_{q-1}$.
\end{conjecture}

We further determine the extremal connected outerplanar graphs which achieve $\lambda_{2,max}$.

\begin{theorem}\label{thm:t2}
    For $n$ even and sufficiently large, among all connected outerplanar graphs on $n=2q$ vertices, the graph maximizing $\lambda_2$ is unique and isomorphic to the graph in Figure~\ref{fig:2qfig}, denoted by $(K_1\vee P_{q-1})\!\!-\!\!(K_1\vee P_{q-1})$, which is constructed by gluing two disjoint copies of the fan graph $K_1\vee P_{q-1}$ via an edge connecting vertices of smallest degree. 
    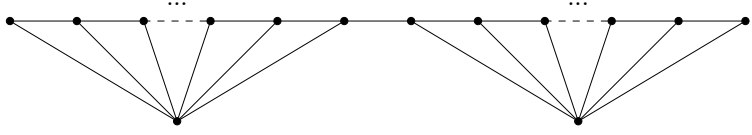
\begin{figure}[h]
     \begin{center}
     \resizebox{10cm}{!}{\begin{tikzpicture}[scale=1, Wvertex/.style={circle, draw=black, fill=white, scale=1}, bvertex/.style={circle, draw=black, fill=black, scale=0.3}]

\node [bvertex] (u0) at (-0.5, -0.5) {};
\node [bvertex] (u1) at (-3, 1) {};
\node [bvertex] (u2) at (-2, 1) {};
\node [bvertex] (u3) at (-1, 1) {};
\node [bvertex] (u4) at (0, 1) {};
\node [bvertex] (u5) at (1, 1) {};
\node [bvertex] (u6) at (2, 1) {};

\node [bvertex] (v0) at (5.5, -0.5) {};
\node [bvertex] (v1) at (3, 1) {};
\node [bvertex] (v2) at (4, 1) {};
\node [bvertex] (v3) at (5, 1) {};
\node [bvertex] (v4) at (6, 1) {};
\node [bvertex] (v5) at (7, 1) {};
\node [bvertex] (v6) at (8, 1) {};

\draw (u0) -- (u1);
\draw (u0) -- (u2);
\draw (u0) -- (u3);
\draw (u0) -- (u4);
\draw (u0) -- (u5);
\draw (u0) -- (u6);

\draw (u1) -- (u2);
\draw (u3) -- (u2); 
\draw [dashed] (u4) -- (u3) node [midway, above=3pt] {$...$};
\draw (u5) -- (u4);
\draw (u6) -- (u5);

\draw (v0) -- (v1);
\draw (v0) -- (v2);
\draw (v0) -- (v3);
\draw (v0) -- (v4);
\draw (v0) -- (v5);
\draw (v0) -- (v6);

\draw (v1) -- (v2);
\draw (v3) -- (v2); 
\draw [dashed] (v4) -- (v3) node [midway, above=3pt] {$...$};
\draw (v5) -- (v4);
\draw (v6) -- (v5);

\draw (u6) -- (v1);

\end{tikzpicture}}
     \end{center}
     \caption{The graph $(K_1\vee P_{q-1})\!\!-\!\!(K_1\vee P_{q-1})$.}
     \label{fig:2qfig}
     \end{figure}
\end{theorem}

\begin{corollary}\label{cor:c1}
   For $n=2q$ even and sufficiently large, among all (not necessarily connected) outerplanar graphs on $n$ vertices, the maximum $\lambda_2$ is uniquely achieved by $2(K_1\vee P_{q-1})$.
\end{corollary}

\begin{remark}
    For $n=12$, $\lambda_{2, max}$ is not achieved by $(K_1\vee P_{5})\!\!-\!\!(K_1\vee P_{5})$. Instead, it is achieved by the graph in Figure~\ref{fig:12vertexextremal}. We further conjecture that for all even $n\geq 14$, the connected outerplanar graph maximizing $\lambda_2$ is unique and isomorphic to the graph $(K_1\vee P_{q-1})\!\!-\!\!(K_1\vee P_{q-1})$, where $n=2q$. 
    
    As some evidence for this conjecture, note that deleting the bridge in the graph in Figure~\ref{fig:12vertexextremal} leaves two copies of the outerplanar graph on $6$ vertices with maximum spectral radius, whereas Lin and Ning~\cite{Lin-Ning19} showed that the outerplanar graph with maximum spectral radius on $n$ vertices for any $n\ge 7$ is $K_1\vee P_{n-1}$. 
\end{remark}

\begin{figure}
\begin{center}
     \resizebox{10cm}{!}{\begin{tikzpicture}[scale=1, Wvertex/.style={circle, draw=black, fill=white, scale=1}, bvertex/.style={circle, draw=black, fill=black, scale=0.3}]

\node [bvertex] (u0) at (-1, 0) {};
\node [bvertex] (u1) at (-2.5, 1) {};
\node [bvertex] (u2) at (-1.5, 1) {};
\node [bvertex] (u3) at (-0.5, 1) {};
\node [bvertex] (u4) at (0.5, 1) {};

\node [bvertex] (v0) at (4, 0) {};
\node [bvertex] (v1) at (2.5, 1) {};
\node [bvertex] (v2) at (3.5, 1) {};
\node [bvertex] (v3) at (4.5, 1) {};
\node [bvertex] (v4) at (5.5, 1) {};

\node [bvertex] (c0) at (-1, 1.5) {}; 
\node [bvertex] (c2) at (4, 1.5) {};

\draw (u0) -- (u1);
\draw (u0) -- (u2);
\draw (u0) -- (u3);
\draw (u0) -- (u4);

\draw (v0) -- (v1);
\draw (v0) -- (v2);
\draw (v0) -- (v3);
\draw (v0) -- (v4);

\draw (u1) -- (u2);
\draw (u3) -- (u2); 
\draw (u4) -- (u3);

\draw (v1) -- (v2);
\draw (v3) -- (v2); 
\draw (v4) -- (v3);

\draw (c0) -- (c2);

\draw (c0) -- (u2); 
\draw (c0) -- (u3); 

\draw (c2) -- (v2); 
\draw (c2) -- (v3); 

\end{tikzpicture}}
\end{center}
\caption{The outerplanar graph on $12$ vertices with maximum $\lambda_2$.}
\label{fig:12vertexextremal}
\end{figure}
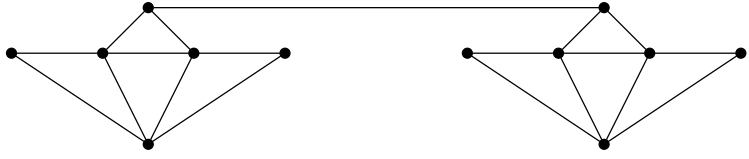

\begin{theorem} \label{thm:t3}
 For $n$ odd and sufficiently large, among all connected outerplanar graphs on $n=2q+1$ vertices, the graphs $G$ maximizing $\lambda_2$ have the following structure: $G$ has a unique cut vertex $u$ and $G\setminus\{u\}$ consists of two disjoint copies of the fan graph $K_1\vee P_{q-1}$. Moreover, for all of these graphs $G$, $\lambda_2(G)=\lambda_1(K_1\vee P_{q-1})$.     
\end{theorem}

Note that the graphs in Theorem~\ref{thm:t3} also maximize $\lambda_2$ over all (not necessarily connected) outerplanar graphs on $n=2q+1$ vertices, with $n$ sufficiently large. The extremal graphs $G$ have the same structure, but do not need to be connected.

We also determine the maximum $\lambda_2$ among all 2-connected outerplanar graphs on $n$ vertices.
\begin{theorem}\label{thm:t4}
    Among all 2-connected outerplanar graphs on $n$ vertices, for $n$ sufficiently large, the maximum $\lambda_2$ is obtained by the graph in Figure~\ref{fig:2conn}, denoted by 
    $(K_1\vee P_{\lfloor \frac{n}{2}\rfloor-1}) \diamond (K_1\vee P_{\lceil \frac{n}{2}\rceil-1})$,
    which is constructed by gluing a fan graph $K_1\vee P_{\lfloor \frac{n}{2}\rfloor-1}$ and another fan graph  $K_1\vee P_{\lceil \frac{n}{2}\rceil-1}$ by connecting the first vertex on the path $P_{\lfloor \frac{n}{2}\rfloor-1}$ to the second vertex on the path $P_{\lceil \frac{n}{2}\rceil-1}$,  and the second vertex on the path $P_{\lfloor \frac{n}{2}\rfloor-1}$ to the first vertex on the path $P_{\lceil \frac{n}{2}\rceil-1}$.
    \begin{figure}[h]
        \begin{center}
     \resizebox{10cm}{!}{\begin{tikzpicture}[scale=1, Wvertex/.style={circle, draw=black, fill=white, scale=1}, bvertex/.style={circle, draw=black, fill=black, scale=0.3}]

\node [bvertex] (u0) at (-0.5, -0.5) {};
\node [bvertex] (u1) at (-3, 1) {};
\node [bvertex] (u2) at (-2, 1) {};
\node [bvertex] (u3) at (-1, 1) {};
\node [bvertex] (u4) at (0, 1) {};
\node [bvertex] (u5) at (1, 1) {};
\node [bvertex] (u6) at (2, 1) {};

\node [bvertex] (v0) at (3.5, 2.5) {};
\node [bvertex] (v1) at (1, 2) {};
\node [bvertex] (v2) at (2, 2) {};
\node [bvertex] (v3) at (3, 2) {};
\node [bvertex] (v4) at (4, 2) {};
\node [bvertex] (v5) at (5, 2) {};
\node [bvertex] (v6) at (6, 2) {};

\draw (u0) -- (u1);
\draw (u0) -- (u2);
\draw (u0) -- (u3);
\draw (u0) -- (u4);
\draw (u0) -- (u5);
\draw (u0) -- (u6);

\draw (u1) -- (u2);
\draw (u3) -- (u2); 
\draw [dashed] (u4) -- (u3) node [midway, above=3pt] {$...$};
\draw (u5) -- (u4);
\draw (u6) -- (u5);

\draw (v0) -- (v1);
\draw (v0) -- (v2);
\draw (v0) -- (v3);
\draw (v0) -- (v4);
\draw (v0) -- (v5);
\draw (v0) -- (v6);

\draw (v1) -- (v2);
\draw (v3) -- (v2); 
\draw [dashed] (v4) -- (v3) node [midway, below=3pt] {$...$};
\draw (v5) -- (v4);
\draw (v6) -- (v5);

\draw (u6) -- (v2);
\draw (u5) -- (v1);
\end{tikzpicture}}
     \end{center} 
     \caption{The graph  $(K_1\vee P_{\lfloor \frac{n}{2}\rfloor-1})\diamond(K_1\vee P_{\lceil \frac{n}{2}\rceil-1})$.}
     \label{fig:2conn}
     \end{figure}
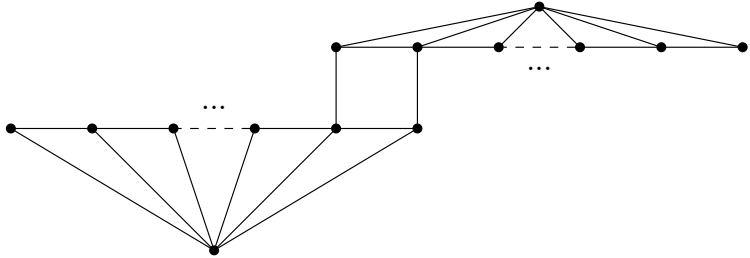
\end{theorem}

The paper is organized as follows. In Section 2, we give necessary definitions and lemmas. In Section 3, we study the coarse structure of the outerplanar graphs which have the maximum $\lambda_k$ and prove Theorem~\ref{thm:t1}. In Section 4, we give the proofs of Theorems \ref{thm:t2} and \ref{thm:t3}. In Section 5, we prove Theorem~\ref{thm:t4}. We conclude in Section 6 with remarks and open problems. 

\section{Notation and lemmas}

For a given graph $G = (V,E)$, the \emph{neighborhood} of a vertex $v \in V$, denoted $N(v)$, is the set of all vertices adjacent to $v$. The \emph{closed neighborhood} of a vertex $v \in V$, denoted $N[v]$, is the set $N(v) \cup \{v\}$. The \emph{degree} of a vertex $v \in V$, denoted $d_v$, is $|N(v)|$. If $G'$ is a subgraph of $G$, then $d_v^{G'}$ is the degree of a vertex $v\in V$ in the subgraph $G'$. We use the notation $G[U]$ for the induced subgraph of $G$ on the vertex set $U\subset V$. If $A, B\subset V$ and $A\cap B = \emptyset$, then $E(A, B)$ is the set of edges in $E$ with one vertex in $A$ and one vertex in $B$. A \emph{walk} in a graph $G$ is a sequence of vertices $v_0v_1 \dots v_{k}$ where $v_iv_{i+1} \in E$ for all $0\leq i \leq k-1$. A \emph{path} in a graph $G$ is a walk where the vertices are distinct. A \emph{linear forest} is a disjoint union of paths. The \emph{path graph} $P_n = (V,E)$ is defined by $V = \left\{ v_0, v_1, \dots v_{n-1}\right\}$ and $E = \left\{ v_iv_{i+1} : 0\leq i \leq n-2\right\}$.


Let us go over some important tools from linear algebra (see~\cite{S10} for additional background on matrix theory).  Assume $A$ is a $n\times n$ real symmetric matrix (or Hermitian matrix in general).  Then the eigenvalues of $A$ are all real, and we may label them in non-increasing order as 
$$\lambda_1 \ge \lambda_2 \ge \cdots \ge  \lambda_n.$$



Recall that a sequence $\mu_1 \geq \dotsm \geq \mu_m$ is said to interlace a sequence $\lambda_1 \geq \dotsm \geq \lambda_n$ with $m<n$ when $\lambda_i \geq \mu_i \geq \lambda_{n-m+i}$ for $i = 1, \ldots, m$. A corollary of Cauchy's Interlacing Theorem states that if $B$ is a principal submatrix of a symmetric matrix $A$, then the eigenvalues of $B$ interlace the eigenvalues of $A$ \cite{H1995}. In particular, the eigenvalues of a proper induced subgraph $H$ of $G$ interlace the eigenvalues of $G$. Throughout this paper, we use $\bz$ for an eigenvector corresponding to $\lambda_1$ and $\bx$ for an eigenvector corresponding to $\lambda_2$. 


Tait and Tobin \cite{TT2017} proved that for sufficiently large $n$, the maximum spectral radius of an outerplanar graph is uniquely achieved by the fan graph $K_1\vee P_{n-1}$. Thus, we have
 $$\lambda_{1,max}=\lambda_1(K_1\vee P_{n-1}).$$

We need an estimate for the size of the largest eigenvalue $\lambda_1$ of the fan graph $K_1 \vee P_{n-1}$. We use a series expansion proven in~\cite{LLLW2022}. 

\begin{lemma}\label{lem:specradfangph}
For a positive integer $n$, let $G$ be the graph $K_1 \vee P_{n-1}$. The largest eigenvalue $\lambda_1$ of $G$ satisfies
\begin{equation}
   \lambda_1 = \sqrt{n-1}+ 1 + \frac{1}{2\sqrt{n-1}} -\frac{1}{n-1} - \frac{1}{8(n-1)^{3/2}} - \frac{7}{16(n-1)^{5/2}} + O\left(\frac{1}{(n-1)^3}\right).
   \label{eq:lambda1_fan_graph}
\end{equation}

\end{lemma}

\begin{proof}
    From \cite{LLLW2022}, $\lambda_1$ satisfies the following equation.
\begin{equation*}
    \lambda_1^2 = (n-1) + \frac{2n-4}{\lambda_1}
    + \frac{4n-10}{\lambda_1^2}
       + \frac{8n-24}{\lambda_1^3}
          + \frac{16n-54}{\lambda_1^4} 
          + \frac{32n-120}{\lambda_1^5}
          + \frac{64n-260}{\lambda_1^6}
          +O\left(\frac{n}{\lambda_1^7}\right).
\end{equation*}
Expanding the largest root of the equation above into a series in terms of $n-1$, we get equation \eqref{eq:lambda1_fan_graph}. 
\end{proof}

\begin{lemma}\label{lem:lamda2even}
Let $G$ be the graph $(K_1\vee P_{q-1})\!\!-\!\!(K_1\vee P_{q-1})$. The second largest eigenvalue $\lambda_2$ of $G$ satisfies the series expansion
    \begin{equation*}
        \lambda_2 =  \sqrt{q-1} + 1 + \frac{1}{2\sqrt{q-1}} - \frac{3}{2(q-1)} + \frac{7}{8(q-1)^{3/2}} - \frac{2}{(q-1)^2} + O\left(\frac{1}{(q-1)^{5/2}}\right).
    \end{equation*}
\end{lemma}

\begin{proof}
    Let $n$ be the total number of vertices in $G$, so that $n = 2q$. We assume $n$ (and $q$) are sufficiently large. Let $u_1$ and $u_2$ be the two largest degree vertices in $G$. We first show that $\lambda_2$ is a simple eigenvalue. By Cauchy's Interlacing Theorem, $\lambda_2(G) \ge \lambda_2((K_1\vee P_{q-1}) \cup (K_1\vee P_{q-2})) = \lambda_1(K_1\vee P_{q-2}) = \sqrt{q-2} + 1 + O\left(\frac{1}{\sqrt{q}}\right)$. On the other hand, we also have by Cauchy's Interlacing Theorem that $\lambda_3(G) \le \lambda_1(G\setminus\{u_1, u_2\}) = \lambda_1(P_{2q-2}) < 2$, so $\lambda_2$ is simple. Let $\bf{z}$ and $\bf{x}$ be the simple eigenvectors associated with $\lambda_1$ and $\lambda_2$, respectively. By the Perron-Frobenius Theorem, $\bf{z}$ is strictly positive. Furthermore, since $A(G)$ is real symmetric, $\langle \mathbf{x}, \mathbf{z}\rangle = 0.$ We show $x_{u_1} = -x_{u_2}$. Assume for a contradiction that $x_{u_1} > -x_{u_2}$. Since $\lambda_2 > 0$, we have $\lambda_2(x_{u_1} + x_{u_2}) > 0$, which implies
    \[\sum_{u \in N(u_1)} x_{u} + \sum_{ u \in N(u_2)} x_{u} > 0. \]
    Therefore,
    \[
    \sum_{u \in V(G)} x_{u} > 0.
    \]
    However, 
    \[
    \langle \mathbf{x}, \mathbf{z} \rangle \geq z_{min} \sum_{u \in V(G)} x_{u} > 0, 
    \]
    which contradicts the fact that $\langle{\mathbf{x}, \mathbf{z}\rangle} = 0$. A similar argument yields a contradiction when $x_{u_1} < -x_{u_2}$. We conclude that $x_{u_1} = - x_{u_2}.$

    Let $P$ be the path on $2q-2$ vertices obtained after deleting the two large degree vertices, and let $A_p$ be the adjacency matrix of $P$. Let ${\bf y}$ be the restriction of the eigenvector ${\bf x}$ on the path $P$. 
    Let $\beta$ be a $(2q-2)$-dimensional column vector with entries equal to $1$ on vertices in $N(u_1)$ and equal to $-1$ on vertices in $N(u_2)$. The eigen-equation gives %
    \begin{align*}
    {\bf y} 
    &= (\lambda_2 I -A_p)^{-1}\beta \\
    &= \frac{1}{\lambda_2}(I-\lambda_2^{-1}A_p)^{-1}\beta \\
    &=\frac{1}{\lambda_2} \sum_{i=0}^\infty \lambda_2^{-i}A_p^i\beta \\
    &=\sum_{i=0}^\infty \lambda_2^{-(i+1)}A_p^i\beta.
    \end{align*}

    This series converges because $\lambda_2 > 2 > \lambda_1(A_p)$. Since $\lambda_2(1) = \sum_{i \in N(u_1)} x_i$ and $\lambda_2(-1) = \sum_{i \in N(u_2)} x_i$, we know that $2\lambda_2 = \beta^T{\bf y}$. Therefore, 

    \begin{equation}
    \label{eq:lambda1_lowerBound}
        \lambda_2 = \frac{1}{2} \sum_{i=0}^\infty \lambda_2^{-(i+1)}\beta^TA_p^i\beta.
    \end{equation}
    When $i=0$, we get \[\beta^T\beta = 2q-2.\]
    When $i=1$, we get
    \begin{equation*}
        \beta^TA_p\beta = 2(2q-2)- 2(2)- 2(1)  = 4q-10.
    \end{equation*}
    When $i=2$, we get
    \begin{equation*}
        \beta^TA_p^2\beta = 4(2q-2) - 6(2) - 2(1) = 8q-22.
    \end{equation*}
    Similarly, for $3\le i\le 5$, we obtain
    \begin{align*}
        \beta^TA_p^3\beta &= 16q-56,\\
        \beta^TA_p^4\beta &= 32q-118,\\
        \beta^TA_p^5\beta &= 64q-272.
    \end{align*}
    After multiplying both sides of \eqref{eq:lambda1_lowerBound} by $\lambda_2$ and simplifying, we obtain that $\lambda_2$ is a root of the equation,
    \begin{equation}\label{eqn:lmb2eqn}
        \lambda_2^2 = (q-1) + \frac{2q-5}{\lambda_2} + \frac{4q-11}{\lambda_2^2} + \frac{8q-28}{\lambda_2^3} + \frac{16q-59}{\lambda_2^4} + \frac{32q-136}{\lambda_2^5} + O\left(\frac{q}{\lambda_2^6}\right).
    \end{equation}
    Using SageMath, we get the following series expansion for $\lambda_2$,
    \begin{equation*}
        \lambda_2 = \sqrt{q-1} + 1 + \frac{1}{2\sqrt{q-1}} - \frac{3}{2(q-1)} + \frac{7}{8(q-1)^{3/2}} - \frac{2}{(q-1)^2} + O\left(\frac{1}{(q-1)^{5/2}}\right).
    \end{equation*}
\end{proof}

\begin{corollary}\label{cor:2lb}
Let $q=\lfloor \frac{n}{2}\rfloor$.
Then we have
$$\lambda_{2,max}\geq \sqrt{q-1} +1 + \frac{1}{2\sqrt{q-1}}
+O\left(\frac{1}{q-1}\right).$$ 
\end{corollary}
For any fixed $k\geq 3$, we get a slightly worse lower bound.
\begin{corollary}\label{cor:klb}
For any fixed $n\geq k\geq 2$, let $q=\lfloor \frac{n-1}{k}\rfloor$.
Then we have
$$\lambda_{k,max}\geq \sqrt{q-1} +1 + \frac{1}{2\sqrt{q-1}}
+O\left(\frac{1}{q-1}\right).$$ 
\end{corollary}
\begin{proof}
Write $n=kq+r$ for $1\leq r\leq k$. 
Let $H$ be a disjoint union of $r-1$ copies of $K_1\vee P_{q}$ and $k-r+1$ copies of $K_1\vee P_{q-1}$. Add a new vertex $u$ to $H$ and for each connected component of $H$ select one vertex and connect it to $u$. Clearly $G$ is an outerplanar graph. The number of vertices in $G$ is given by 
$$(r-1)(q+1)+(k-r+1)q+1= kq+r=n.$$ 
By Cauchy's Interlacing Theorem, we have
    $$ \lambda_{k,max}\geq \lambda_k(G)\geq \lambda_k(H).$$
Since $H$ is the disjoint union of $k$ many fan graphs, $\lambda_k(H)$ is equal to the first eigenvalue of its smallest component.
We have
\begin{align*}
\lambda_{k,max} &\geq \lambda_k(H) \\
&\geq \lambda_{1}(K_1\vee P_{q-1}) \\
&= \sqrt{q-1} + 1 + \frac{1}{2\sqrt{q-1}}+ O\left(\frac{1}{q-1}\right).
\end{align*}
Here we apply Lemma \ref{lem:specradfangph}.
\end{proof}

\begin{lemma} \label{lem:P3}
    For any outerplanar graph $G$, let $h_i(u,v)$ denote the number of $(u,v)$-paths of length i. Then for all $u,v \in V(G)$ we have
    \begin{align*}
         h_2(u,v)&\leq 2,\\
         h_3(u,v)&\leq 8,\\
         h_4(u,v)&\leq 98.\\
    \end{align*}
\end{lemma}

The bounds on $h_3(u, v)$ and $h_4(u, v)$ in Lemma~\ref{lem:P3} are certainly not tight, but we use them as we only need that $h_3(u, v)$ and $h_4(u, v)$ are bounded by a constant independent of the number of vertices $n$. 

\begin{proof}
    Let $G$ be an outerplanar graph and $u,v \in V(G)$. Since $G$ is $K_{2,3}$-minor-free, there cannot be three or more internally disjoint $(u, v)$-paths in $G$. This proves $h_2(u, v) \le 2$. Let $uijv$ be a $(u, v)$-path of length $3$. There can be at most one other $(u, v)$-path of length $3$ which is internally disjoint from $uijv$, as otherwise $G$ would contain a $K_{2, 3}$-minor. Every other $(u, v)$-path of length $3$ contains either $i$ or $j$. Since $G$ does not contain a $K_{2, 3}$, there are at most one other path of the form $ui\star v$ (here $\star$ can be any vertex); at most two paths of the form $uj\star v$; at most two paths of the form $u\star iv$; and at most one other path of the form $u\star jv$. This gives $h_3(u, v) \le 1 + 1 + 1 + 2 + 2 + 1 = 8$. Now, let $uabcv$ be a $(u, v)$-path of length $4$. There is at most one $(u, v)$-path of length $4$ which is internally disjoint from $uabcv$, as then $G$ would contain a $K_{2, 3}$-minor. Using the fact that $G$ does not contain a $K_{2, 3}$, we now observe that: there are at most $8$ $(u, v)$-paths of length $4$ of the form $ua\star\star v$ (and similarly for $ub\star\star v$, $uc\star\star v$, $u\star\star av$, $u\star\star bv$ and $u\star\star cv$); at most $4$ paths of the form $u\star a\star v$ (and similarly for $u\star b\star v$ and $u\star c\star v$); and at most $2$ paths for each of the $18$ possibilities with two of $a$, $b$, and $c$ on the path. This gives $h_4(u, v)\le 1 + 1 + 8(6) + 4(3) + 2(18) = 98$. 
\end{proof}

\section{Proof of Theorem \ref{thm:t1}}
Let $G$ be a connected outerplanar graph on $n$ vertices with $k$th largest eigenvalue equal to $\lambda_{k,max}$.  We denote the adjacency matrix of $G$ by $A(G)$.  The aim of this section is to prove Theorem~\ref{thm:t1}.

\begin{lemma}
   Suppose $G$ is an outerplanar graph on $n$ vertices with maximum degree $\Delta$. Then we have
   \begin{equation*}
       \Delta\geq \lambda_1^2-\frac{4(2n-3)}{\lambda_1}.
   \end{equation*}
\end{lemma}
\begin{proof}
Let $\bz$ be an eigenvector corresponding to $\lambda_1$. Without loss of generality, we normalize $\bz$ such that $\left\lVert \bz \right\rVert_\infty =1$. Let $u_0 \in V(G)$ be such that $z_{u_0} = 1$.
For any vertex $u$, we have
\[\lambda_1 z_u = \sum_{v\in N(u)}z_v  \le d_{u}. \]
Since any outerplanar graph on $n$ vertices has at most $2n-3$ edges, it follows by summing over all vertices $u\in V(G)$ that
\[\lambda_1\sum_{u\in V(G)}z_u \le \sum_{u\in V(G)} d_u \le 2(2n-3).\]
This implies
$$ \sum_{u\in V(G)}z_u \leq \frac{2(2n-3)}{\lambda_1}.$$
We observe,
\begin{align*}
\lambda_1^2
&= \lambda_1^2 z_{u_0}\\
&= \sum_{v \in N(u_0) } \lambda_1z_v \\
&=  \sum_{v \in N(u_0)} \sum_{w\in N(v)} z_w \\
&= d_{u_0} + \sum_{w\in V\setminus \{u_0\}} z_w |N(w)\cap N(u_0)|.
\end{align*}
Since $G$ is $K_{2,3}$-free,  we have $|N(w)\cap N(u_0)|\leq 2$.
Therefore,
\begin{align*}
\Delta&\geq d_{u_0}\\
&=\lambda_1^2  - \sum_{w\in V\setminus \{u_0\}} z_w |N(w)\cap N(u_0)|\\
&\geq \lambda_1^2 - 2\sum_{u\in V(G)}z_u \\
&\geq \lambda_1^2 - \frac{4(2n-3)}{\lambda_1}.
\end{align*}
\end{proof}

\begin{lemma}\label{lem:u0largedeg} 
There exist $k$ vertices $u_1, \ldots, u_k$ such that  for $1\leq i\leq k$,
\[d_{u_i} \ge \frac{n}{k} - O(\sqrt{n}).\]
\end{lemma}

\begin{proof}
We will find large degree vertices $u_1, u_2,\ldots, u_k$ in sequence.

Since $\lambda_1(G)\geq \lambda_k(G)\geq  \sqrt{n/k} +1 + O(\frac{1}{\sqrt{n}})$,
we have
$$\Delta(G)\geq \lambda_1^2- \frac{4(2n-3)}{\lambda_1} \geq \frac{n}{k}-O(\sqrt{n}).$$

Let $u_1$ be the vertex with the largest degree in $G$. Now let $G_1=G-\{u_1\}$. By Cauchy's Interlacing Theorem, we have
$\lambda_1(G_1)\geq \lambda_2(G)\geq \lambda_k(G)\geq \sqrt{n/k} +1 + O(\frac{1}{\sqrt{n}}).$
Repeat this argument with $G_1$. Let $u_2$ be the vertex with the largest degree in $G_1$. We have
$$ \Delta(G_1)\geq \lambda_1^2(G_1)- \frac{4(2n-3)}{\lambda_1(G_1)} \geq \frac{n}{k}-O(\sqrt{n}).$$

In general, assume we have already found vertices $u_1,u_2,\ldots, u_i$, for some $i<k$. Consider
$G_{i}= G\setminus \{u_1, u_2, \ldots, u_i\}$. We have
$\lambda_1(G_i)\geq \lambda_{i+1}(G)\geq \lambda_k(G)\geq \sqrt{n/k} +1 + O(\frac{1}{\sqrt{n}}).$
Thus, there is a vertex $u_{i+1}$ with degree
\[d_{u_{i+1}} =\Delta(G_i) \ge 
\lambda_1^2(G_i)- \frac{4(2n-3)}{\lambda_1(G_i)} \geq
\frac{n}{k} - O(\sqrt{n}).\]
\end{proof}

Let $U=\{u_1, u_2,\ldots, u_k\}$ be the set of vertices in $G$ with degree at least
$\frac{n}{k}-O(\sqrt{n})$. We will show that all other vertices not in $U$ have small degree in $G$.

\begin{lemma}
For any other vertex $u \not\in \{u_1,\ldots, u_k\}$, we have
$$d_u=O(\sqrt{n}).$$
\end{lemma}
\begin{proof}
    First we show that the union of the neighbors of $u_1,\ldots, u_k$ covers almost all vertices in $G$.
    We have
    \begin{align*}
        \left|\bigcup_{i=1}^k N(u_i)\right |& \geq \sum_{i=1}^k |N(u_i)| - \sum_{1\leq i<j\leq k} |N(u_i)\cap N(u_j)|\\
        & \geq \sum_{i=1}^k d_{u_i} -2{k\choose 2}\\
        & \geq n- O(\sqrt{n}).
    \end{align*}
This implies
$$ \left|\bigcap_{i=1}^k \overline{ N(u_i)}\right | = O(\sqrt{n}).$$
  For any $u\not\in \{u_1,\ldots, u_k\}$, $u$ can have at most 2 neighbors in $N(u_i)$. Thus,
$$d_u\leq 2k + \left|\bigcap_{i=1}^k \overline{ N(u_i)}\right | =O(\sqrt{n}).$$
\end{proof}

 For $1\leq i\leq k$, let $\tilde d_{u_i}$ be the number of neighbors of $u_i$ in $V(G)\setminus U$, i.e., 
 \[\tilde d_{u_i}=\left| N(u_i)\setminus U\right |.\]

\begin{lemma}
 For sufficiently large $n$ and any $u\in U$, we have $\tilde d_u\geq \frac{n}{k}-O(1).$
\end{lemma}

\begin{proof}
Without loss of generality, we assume that 
 $$\tilde d_{u_1}\geq \tilde d_{u_2}\geq \cdots \geq \tilde d_{u_k}.$$
Let $G'=G\setminus \{u_1,\ldots, u_{k-1}\}$. Then $d^{G'}_{u_k}=\tilde d_{u_k}$.
In particular, $G'$ has one unique vertex $u_k$ with degree at least $\frac{n}{k}-O(\sqrt{n})$ while all other vertices have degree at most $O(\sqrt{n})$. By Cauchy's Interlacing Theorem, we have
  $$\lambda_1(G')\geq \lambda_{k}(G)\geq \sqrt{\frac{n}{k}}+1 + O\left(\frac{1}{\sqrt{n}}\right).$$

For the rest of the proof, all notation is relative to $G'$ unless stated otherwise.
  Let $\bz$ be an eigenvector corresponding to $\lambda_1$ for $G'$, which is normalized such that $\left\lVert \bz \right\rVert_\infty =
  z_{u'} =1$ for some vertex $u'$. For any $u\in V(G')\setminus \{u_k\}$, we have
  \begin{align*}
  (\lambda_1^2 -d_{u})z_u &= \sum_{w\in V\setminus \{u\}} z_w |N(w)\cap N(u)| \\
  &\leq 2\sum_{w\in V\setminus \{u\}}z_w \\
  &\leq \frac{4(|E(G')|)}{\lambda_1}.
  \end{align*}
  Since $d_u=O(\sqrt{n})$, we have
  \begin{equation}\label{eqn:xuub}
  z_u\leq \frac{4|E(G')|}{\lambda_1(\lambda_1^2-d_u)} = O\left(\frac{1}{\sqrt{n}}\right). 
  \end{equation}
Thus, $u'=u_k$. Now,
$$\lambda_1^2 -d_{u_k}=(\lambda_1^2-d_{u_k})z_{u_k}=\sum_{w\in V\setminus \{u_k\}} z_w |N(w)\cap N(u_k)|. $$
Multiplying by $\lambda_1$ on both sides, we have
\begin{align}
    \lambda_1(\lambda_1^2 -d_{u_k}) &= \sum_{w\in V\setminus \{u_k\}} \lambda_1 z_w |N(w)\cap N(u_k)|
    \nonumber\\
    &= 2t_3(u_k)+
    \sum_{v\in N(u_k)} z_v (d_v-1) + \sum_{y\in V\setminus \{u_k\}} z_y h_3(u_k,y). \label{eq:h3}
\end{align}
Here $t_3(u_k)$ is the number of triangles containing $u_k$, while $h_3(u_k,y)$ denotes the number of $(u_k, y)$-paths of length 3.  We have 
$$2t_3(u_k)\leq 2(d_{u_k}-1).$$
By Lemma \ref{lem:P3}, we have
$$h_3(u_k,y) \leq 8.$$
Since $z_u = O(\frac{1}{\sqrt{n}})$ for all $u\in V(G')\setminus \{u_k\}$, we have
\begin{align*}
    \sum_{v\in N(u_k)} z_v d_v &\leq O\left(\frac{1}{\sqrt{n}}\right)\sum_{v\not=u_k} d_v\\
    &\leq  O\left(\frac{1}{\sqrt{n}}\right) 2|E(G')|\\ 
    &=O(\sqrt{n}),
\end{align*}
and
$$\sum_{y\in V\setminus \{u_k\}} z_y h_3(u_k,y) \leq 8 \sum_{y} z_y \leq O(\sqrt{n}).$$
Plugging into inequality \eqref{eq:h3}, we have
\begin{equation}\label{eq:cubic}
    \lambda_1(\lambda_1^2-d_{u_k})\leq 2d_{u_k} + O(\sqrt{n}).
\end{equation}
Therefore,
\begin{align*}
    d_{u_k} &\geq \frac{\lambda_1^3- O(\sqrt{n})}{\lambda_1+2}\\
    &= (\lambda_1 -1)^2 -O(1)\\
    &\geq \frac{n}{k}-O(1).
\end{align*}
Thus, we proved $d_{u_k}\geq \frac{n}{k}-O(1)$. By the construction of $G'$, we conclude that the $k$ largest degree vertices each have degree at least $\frac{n}{k}-O(1)$. Consequently, all vertices except for these $k$ vertices have degree $O(1)$.
\end{proof}

\begin{proof}[Proof of Theorem \ref{thm:t1}]
Rename the vertices $u_1, \ldots, u_k$ such that 
$$d_{u_1}\geq d_{u_2}\geq \cdots \geq d_{u_k}.$$
We have
\begin{align*}
d_{u_k}&\leq \frac{1}{k} \sum_{i=1}^k |N({u_i})| \\
&\leq \frac{1}{k} \left( \left| \bigcup_{i=1}^n N({u_i}) \right| 
+ \sum_{1\leq i<j\leq k} \left| N(u_i)\cap N(u_j)  \right| \right )
\\
&\leq \frac{1}{k}(n+k^2-k).
\end{align*}
This implies
$$d^{G'}_{u_k}\leq d_{u_k}\leq \frac{n}{k}+k-1.$$
From inequality \eqref{eq:cubic}, we have
\begin{align*}
    \lambda_1(G')^2 &\leq d^{G'}_{u_k} + \frac{2d^{G'}_{u_k} + O(\sqrt{n})}{\lambda_1} \\
    &\leq \frac{n}{k}+k-1 + 2\sqrt{\frac{n}{k}}  +O(1)\\
    &= \frac{n}{k} + 2\sqrt{\frac{n}{k}}  +O(1).
\end{align*}
This implies
$$\lambda_1(G') \leq \sqrt{ \frac{n}{k} + 2\sqrt{\frac{n}{k}}  +O(1) } = \sqrt{n/k} + 1+  O\left(\frac{1}{\sqrt{n}}\right).
$$
By Cauchy's Interlacing Theorem, we have
$$\lambda_k(G) \leq \lambda_1(G') \leq \sqrt{n/k} + 1 + O\left(\frac{1}{\sqrt{n}}\right).$$
This matches the lower bound in Corollary~\ref{cor:klb} asymptotically.
\end{proof}

The following lemma will be useful in determining the extremal graphs for small $k$.

\begin{lemma} \label{lem:degrees} 
\begin{enumerate}
    \item For $k\geq 3$, sufficiently large $n$, and any $u\in U$, we have $$\tilde d_u\geq \left\lfloor \frac{n-1}{k} \right\rfloor -1.$$
  \item For $k=2$, sufficiently large $n$, and any $u\in U$, we have $$\tilde d_u\geq \left\lfloor \frac{n}{k} \right\rfloor -1.$$    
\end{enumerate}

\end{lemma}

\begin{proof}
Let $q=\lfloor \frac{n-1}{k} \rfloor$ when $k\geq 3$ and
$q=\lfloor \frac{n}{k} \rfloor$ when $k=2$.
  Without loss of generality, we assume that 
  $$\tilde d_{u_1}\geq \tilde d_{u_2}\geq \cdots \geq \tilde d_{u_k}.$$
Let $G'=G\setminus \{u_1,\ldots, u_{k-1}\}$. Then $d^{G'}_{u_k}=\tilde d_{u_k}$. All notation below is related to the graph $G'$ if not specified.
From equation \eqref{eq:h3}, we have
\[
    \lambda_1(\lambda_1^2 -d_{u_k}) = 2t_3(u_k)+
    \sum_{v\in N(u_k)} z_v (d_v-1) + \sum_{x\in V\setminus \{u_k\}} z_x h_3(u_k,x). \label{eq:h4}
\]
Thus, we have
\begin{align*}
    \lambda_1^2  (\lambda_1^2 -d_{u_k}) &= 2t_3(u_k)\lambda_1+
    \sum_{v\in N(u_k)} \lambda_1 z_v (d_v-1) + \sum_{x\in V\setminus \{u_k\}} \lambda_1z_x h_3(u_k,x)\\
   &= 2t_3(u_k)\lambda_1+  \sum_{v\in N(u_k)} \sum_{x\in N(v)} z_x (d_v-1) + \sum_{x\in V\setminus \{u_k\}} \sum_{w\in N(x)}z_w h_3(u_k,x)\\
   &= 2t_3(u_k)\lambda_1  + \sum_{v\in N(u_k)}(d_v-1) + \sum_{v\in N(u_k), x\in N(v)\setminus\{u_k\}}z_x (d_v-1)
   + 2c_4(u_k) +   
   \sum_{w\not=u_k}z_w h_4(u_k, w)\\
   &\hspace*{4mm} + \sum_{v\in N(u_k)} 2 z_v t_3(v) + \sum_{x\not=u_k}z_xh_2(u_k, x)(d_x-1).
\end{align*}

Since for any $v\not= u_k$, $d_v=O(1)$ and $h_i(u_k,v)=O(1)$ for $2\le i\le 4$ by Lemma~\ref{lem:P3}, we have the following estimates for the lower order terms by \eqref{eqn:xuub}:
\begin{align*}
    \sum_{v\in N(u_k), x\in N(v)-u_k}z_x (d_v-1) &= O\left(\sum_x z_x\right)= O(\sqrt{n}).\\
    \sum_{w\not=u_k}z_w h_4(u_k, w) &= O\left(\sum_w z_w\right)= O(\sqrt{n}).\\
    \sum_{v\in N(u_k)} 2z_v t_3(v)&=O\left(\sum_v z_v\right)= O(\sqrt{n}).\\
    \sum_{x\not=u_k}z_x h_2(u_k, x)(d_x-1)&=O\left(\sum_x z_x\right)= O(\sqrt{n}).
\end{align*}
We also use the following two estimates for the terms $t_3(u_k)$ and $c_4(u_k)$.
\begin{align*}
    2t_3(u_k)&\leq 2 d_{u_k} -2,\\
    2c_4(u_k)&\leq 2 d_{u_k} + O(1).
\end{align*}

The estimate for $t_3(u_k)$ follows from the fact that $N(u_k)$ is a linear forest. For the $c_4(u_k)$ estimate, there are two types of 4-cycles containing $u_k$, those whose vertices are all in $N(u_k)$ and those having one vertex not in $N(u_k)$. The first type contributes $2 d_{u_k} + O(1)$ to $2c_4(u_k)$ since $N(u_k)$ is a linear forest. The second contributes $O(1)$. To see this, consider \begin{align*}E(V(G') \setminus N(u_k), N(u_k)) &\leq \sum_{j=1}^{k-1} E(N(u_j), N(u_k)) + E(V(G')\setminus \cup_{j=1}^{k-1} N(u_j), N(u_k))\\
&\leq 3(k-1) + 2|V(G)\setminus \cup_{j=1}^{k} N(u_j)|\\ &= O(1).\end{align*}
Now,
\begin{align*}
    \sum_{v\in N(u_k)}(d_v-1) + 2c_4(u_k) &\leq \sum_{v\in N(u_k)}(d_v-1) + 2 d_{u_k}  +O(1).\\
    &\le 2|E(G'[{N(u_k)\cup\{u_k\}}])| + O(1)\\
    &\leq 4 d_{u_k} + O(1).
\end{align*}
Putting everything together, we have
\begin{equation*}
  \lambda_1^2  (\lambda_1^2 -d_{u_k})  \leq \lambda_1 (2 d_{u_k}-2)+ 4 d_{u_k} + O(\sqrt{n}),
\end{equation*}
and so,
\begin{align*}
    d_{u_k}&\geq \frac{\lambda_1^4+2\lambda_1+O(\sqrt{n})}{\lambda_1^2+2\lambda_1+4}\\
    &= \lambda_1^2-2\lambda_1 -O\left(\frac{1}{\sqrt{n}}\right)\\
    &= (\lambda_1-1)^2-1 -O\left(\frac{1}{\sqrt{n}}\right)\\
    &\geq\left(\sqrt{q-1}+\frac{1}{2\sqrt{q-1}} + O\left(\frac{1}{(q-1)}\right)\right)^2 - 1 -O\left(\frac{1}{\sqrt{n}}\right)\\
    &=q-1 -O\left(\frac{1}{\sqrt{q-1}}\right) - O\left(\frac{1}{\sqrt{n}}\right).
\end{align*}
Here we applied the lower bound of $\lambda_k$ as in Corollary 
\ref{cor:2lb} or \ref{cor:klb}.
Since $d_{u_k}$ is an integer, we have 
\[d_{u_k}\geq q-1.\]
Hence,
\[\tilde d_{u_1}\geq \tilde d_{u_2}\geq \cdots \geq \tilde d_{u_k}\geq q - 1.\]
\end{proof}

\section{Outerplanar graphs with maximum \texorpdfstring{$\lambda_2$}{}: the exact result}

Let $G$ be a connected outerplanar graph on $n$ vertices that maximizes $\lambda_2$. Let $\bx=(x_1,x_2,\ldots,x_n)$ be the eigenvector for $\lambda_2$ (note that $\lambda_2$ is a simple eigenvalue because, by Theorem~\ref{thm:t1}, $\lambda_2 = \sqrt{n/2}+1+O(1/\sqrt{n})$, while in the graph $G$, $\lambda_3 \le \sqrt{n/3} + 1 + O(1/\sqrt{n})$). 
By Lemma \ref{lem:degrees}, $G$ contains exactly two vertices, say $u_1$ and $u_2$, with degree at least $\lfloor \frac{n}{2}\rfloor - 1$.
Since $G$ is outerplanar, $|N(u_1)\cap N(u_2)|\leq 2$. This implies that there are at most $O(1)$ vertices not in $N(u_1)\cup N(u_2)$. Thus, all vertices other than $u_1$ and $u_2$ have degree $O(1)$.
Let $V^+=\{v\in V(G)\colon x_v>0\}$, $V^0=\{v\in V(G)\colon x_v=0\}$,
and $V^-=\{v\in V(G)\colon x_v<0\}$.
For any vertex set $S$, the volume of $S$, denoted by $\vol(S)$, is defined as
$\sum_{v\in S} |x_v|$. For any vertex $v$, define $N^+(v)=N(v)\cap V^+$,
$d^+_v= |N^+(v)|$,
$N^-(v)=N(v)\cap V^-$
and 
$d^-_v=|N^-(v)|$. Let $x^+_{max}=\max \{x_v\colon v\in V^+\}$ and
$x^-_{min}=\min \{x_v\colon v\in V^-\}$.

We now give bounds on $\vol(V^{+})$ and $\vol(V^{-})$ and also show that $x_{max}^+$ and $x_{min}^-$ are achieved at $u_1$ and $u_2$. 
\begin{lemma}\label{lem:volubs}
We have
\begin{enumerate}
    \item For any $v\in V^+$, $\lambda_2x_v\leq d^+_v|x^+_{max}|$. 
    \item For any $v\in V^-$, $\lambda_2|x_v|\leq d^-_v|x^-_{min}|$.
    \item $\vol(V^+)= O(\sqrt{n})x^+_{max}$ and $\vol(V^-)= O(\sqrt{n})|x^-_{min}|$. 
    \item $x^+_{max}$ and $x^-_{min}$ are achieved at $u_1$ and $u_2$.
\end{enumerate}
\end{lemma}
\begin{proof}

    For any vertex $v\in V^+$, we have
\[\lambda_2 x_v = \sum_{u\in N(v)}x_u  \le  \sum_{u\in N^+(v)} x_u \leq
d^+_{v} x^+_{max}. \]
By symmetry, we have for any $v\in V^-$, $\lambda_2|x_v|\leq d^-_v|x^-_{min}|$.

Since any outerplanar graph on $n$ vertices has at most $2n-3$ edges, it follows by summing over all vertices $v\in V^+$ that
\[\lambda_2 \vol(V^+)=
\lambda_2\sum_{v\in V^+}x_v \le x^+_{max}\sum_{v\in V^+} d_v \le 2(2n-3) x^+_{max}.\]
This implies
$$  \vol(V^+)\leq \frac{2(2n-3)}{\lambda_2} x^+_{max} = O\left(\sqrt{n}\right)x^+_{max}.$$
A similar argument implies that \[\vol(V^{-}) = O\left(\sqrt{n}\right) |x_{min}^{-}|.\]
Assume $x^+_{max}$ is achieved at some vertex $u_0$. We have
\begin{align}
\lambda_2^2 x_{u_0}
&\leq \sum_{v \in N^+(u_0) } \lambda_2 x_v  \nonumber
\\
&\leq \sum_{v \in N^+(u_0)} \sum_{w\in N^+(v)} x_w \nonumber\\
&= d^+_{u_0} x_{u_0}+ \sum_{w\in V^+\setminus \{u_0\}} x_w |N^+(w)\cap N^+(u_0)|. \label{eq:quadratic}
\end{align}
Since $G$ is $K_{2,3}$-free,  we have $|N(w)\cap N(u_0)|\leq 2$.
Therefore,
\begin{align*}
d^+_{u_0} x_{u_0}
&=\lambda_2^2 x_{u_0} - \sum_{w\in V^+\setminus \{u_0\}} x_w |N^+(w)\cap N^+(u_0)|\\
&\geq \lambda_2^2 x_{u_0} - 2\sum_{u\in V^+}|x_u| \\
&\geq \lambda_2^2 x_{u_0} - \frac{4(2n-3)}{\lambda_2}x_{u_0}.
\end{align*}
Therefore,
$$d^+_{u_0}\geq \lambda_2^2 -\frac{4(2n-3)}{\lambda_2} =\frac{n}{2}-O(\sqrt{n}), $$
which implies that $u_0$ must be one of the vertices $u_1$ and $u_2$. A similar argument shows that $x_{min}^{-}$ is achieved at either $u_1$ or $u_2$.
\end{proof}

\begin{lemma}\label{lem:u1u2notedge}
We have $u_1u_2\notin E(G)$.
\end{lemma}

\begin{proof}
    We assume for a contradiction that $u_1u_2 \in E(G)$. Without loss of generality, we assume that $x_{u_1} = 1$ and $|x_{u_2}| \ge |x_{u_1}|$. By the fourth claim in Lemma~\ref{lem:volubs}, $x_{u_2}$ is negative, so 
    \[\lambda_2x_{u_1} \leq \sum_{v\in N^{+}(u_1)}x_v - |x_{u_2}|.\]
    Multiplying by $\lambda_2$ and using the same steps as in equation \eqref{eq:quadratic}, we have
    \[(\lambda_2^2 - d_{u_1}^{+})\le \sum_{v\in V^{+}}x_vh_2(u_1, v) - \lambda_2|x_{u_2}|.\]
    Multiplying by $\lambda_2$ again, we obtain 
    \[\lambda_2(\lambda_2^2 - d_{u_1}^{+}) \le 2t_3^{+}(u_1) + \sum_{v\in N^{+}(u_1)}x_v(d_v-1) + \sum_{w\in V^{+}}x_wh_2(u_1, w) - \lambda_2^2,\]
    where $t_3^{+}(u_1)$ is the number of triangles containing $u_1$ with all vertices in the triangle in $V^+$. 
    Therefore, 
    \[\lambda_2(\lambda_2^2 - d_{u_1}^{+}) \le 2d_{u_1}^{+} - \lambda_2^2 + O(\sqrt{n}), \]
    as $\sum_{v\in N^{+}(u_1)}x_v(d_v-1) = O(\sqrt{n})$ and $\sum_{w\in V^{+}}x_wh_2(u_1, w) = O(\sqrt{n})$. 
    Rearranging the inequality, we obtain
    \[(\lambda_2+2)d_{u_1}^{+} \ge \lambda_2^3 + \lambda_2^2 - O(\sqrt{n}),\]
    and so,
    \begin{align*}
        d_{u_1}^{+} &\ge \frac{\lambda_2^3+\lambda_2^2}{\lambda_2+2} - O(1)\\
        &=\lambda_2^2 - \lambda_2 - O(1)\\
        &=\left(\sqrt{q-1}+1 + O\left(\frac{1}{\sqrt{q-1}}\right)\right)^2 - \sqrt{q-1} - O(1)\\
        &= q - 1 + \sqrt{q-1} - O(1),
    \end{align*}
    which is a contradiction to the fact that $d_{u_1}^{+} \le d_{u_1} = q-1 + O(1)$. 
\end{proof}

 Let $P$ be the induced subgraph obtained from $G$ after deleting $u_1$ and $u_2$. Let $\bx$ be the eigenvector of $G$ corresponding to $\lambda_2$. Let $A_P$ be the adjacency matrix of $P$. 
Let $\by$ be the restriction of $\bx$ on the vertex set of $P$.
Define $(n-2)$-dimensional column vectors $\beta_1$, $\beta_2$, $\gamma_1$, and $\gamma_2$, indexed by the vertices of $P$, as follows:
\begin{align}
  \beta_1(v) &=\begin{cases}
    x_{u_1} & \mbox{ if } v\in N(u_1),\\
    0 & \mbox{ otherwise.} 
    \label{eq:beta1def}
    \end{cases}\\
  \beta_2(v) &=\begin{cases}
    x_{u_2} & \mbox{ if } v\in N(u_2),\\
    0 & \mbox{ otherwise.} 
    \label{eq:beta2def}
    \end{cases}\\
   \gamma_1(v) &=\begin{cases}
    \frac{1}{x_{u_1}} & \mbox{ if } v\in N(u_1),\\
    0 & \mbox{ otherwise.} 
    \label{eq:gamma1def}
    \end{cases}\\
  \gamma_2(v) &=\begin{cases}
    \frac{1}{x_{u_2}} & \mbox{ if } v\in N(u_2),\\
    0 & \mbox{ otherwise.} 
    \label{eq:gamma2def}
    \end{cases} 
\end{align}
Let $\beta=\beta_1+\beta_2$ and $\gamma=\gamma_1+\gamma_2$, so that
\begin{align}
  \beta(v)=\begin{cases}
    x_{u_1} & \mbox{ if } v\in N(u_1)\setminus N(u_2),\\
    x_{u_2} & \mbox{ if } v\in N(u_2)\setminus N(u_1),\\
    x_{u_1}+ x_{u_2} & \mbox{ if } v\in N(u_1)\cap N(u_2),\\
    0 & \mbox{ otherwise.}
\end{cases}\\
\gamma(v)=\begin{cases}
    \frac{1}{x_{u_1}} & \mbox{ if } v\in N(u_1)\setminus N(u_2),\\
    \frac{1}{x_{u_2}} & \mbox{ if } v\in N(u_2)\setminus N(u_1),\\
    \frac{1}{x_{u_1}}+ \frac{1}{x_{u_2}} & \mbox{ if } v\in N(u_1)\cap N(u_2),\\
    0 & \mbox{ otherwise.}
\end{cases}
\end{align}
From the eigen-equation for
$\lambda_2$, we have
\[\lambda_2\by = A_P\by + \beta.\]
Note that $\lambda_2$ is not an eigenvalue of $A_P$, as $\lambda_2 = \sqrt{n/2} + 1 + O(1/\sqrt{n})$, while $\lambda_1(A_P) \le \Delta(P) = O(1)$. Therefore, we have
    \begin{align*}
    \by
    &= (\lambda_2 I -A_P)^{-1}\beta \\
    &= \frac{1}{\lambda_2}(I-\lambda_2^{-1}A_P)^{-1}\beta \\
    &=\frac{1}{\lambda_2} \sum_{i=0}^\infty \lambda_2^{-i}A_P^i\beta \\
    &=\sum_{i=0}^\infty \lambda_2^{-(i+1)}A_P^i\beta.
    \end{align*}
Also, by the eigen-equation, we have
\begin{align*}
    \lambda_2 &=x_{u_1}^{-1}\sum_{v\in N(u_1)} x_v = \gamma_1^T \by, \\
    \lambda_2 &=x_{u_2}^{-1}\sum_{v\in N(u_2)} x_v = \gamma_2^T \by.\\
\end{align*}
Therefore, 
    \begin{align}
    \lambda_2^2 &=\sum_{i=0}^\infty \lambda_2^{-i}\gamma_1^T A_P^i\beta,
    \label{eq:gamma1}\\
    \lambda_2^2 &=\sum_{i=0}^\infty \lambda_2^{-i}\gamma_2^T A_P^i\beta.
     \label{eq:gamma2}
    \end{align}   
Taking the average, we get    
    \begin{equation}
    \lambda_2^2 =\frac{1}{2}\sum_{i=0}^\infty \lambda_2^{-i}\gamma^T A_P^i\beta.
    \label{eq:gamma}
    \end{equation}

For $i=0,1,2,\ldots,$ let $a_i=\frac{1}{2} \gamma^T A_P^i\beta$. 
For any two vertices $u$ and $v$ (with $u=v$ allowed), let $w_i(u,v)$ denote the number of $(u, v)$-walks of length $i$. Then we have
\begin{equation} 
\label{eqn:aieqn}
\gamma^T A_P^i \beta
=\sum_{(u,v) \in N(u_1)^2} w_i(u,v) +
\sum_{(u,v) \in N(u_2)^2} w_i(u,v) \\
+ \sum_{(u,v) \in N(u_1) \times N(u_2)} w_i(u,v) \left(\frac{x_{u_2}}{x_{u_1}}
+ \frac{x_{u_1}}{x_{u_2}}\right).
\end{equation}
Since $x_{u_1}$ and $x_{u_2}$ have opposite signs, we have
\[\frac{x_{u_2}}{x_{u_1}}
+ \frac{x_{u_1}}{x_{u_2}} \leq -2.\]
Thus, we have
\begin{equation}\label{eqn:aiineq}
a_i \leq \frac{1}{2}\sum_{(u,v)\in N(u_1)^2} w_i(u,v) + \frac{1}{2}
\sum_{(u,v)\in N(u_2)^2} w_i(u,v)
-\sum_{(u,v)\in N(u_1)\times N(u_2)} w_i(u,v).
\end{equation}
In particular,
\begin{align}
    a_0&= \frac{1}{2} \left(d_{u_1}+d_{u_2}+ |N(u_1)\cap N(u_2)| \left(\frac{x_{u_2}}{x_{u_1}}
+ \frac{x_{u_1}}{x_{u_2}}\right)\right)\nonumber\\
&\leq \frac{1}{2} \left(d_{u_1}+d_{u_2}-2|N(u_1)\cap N(u_2)|\right)\nonumber\\ 
&= \frac{1}{2} |N(u_1)\Delta N(u_2)|.\label{a_0ub}
\end{align}

\begin{lemma}\label{lem:series}
Consider the equation
\begin{equation*}
    \lambda^2 = a_0 +  \dss_{i=1}^{\infty} \frac{a_i}{\lambda^i}.
\end{equation*}
 The largest root $\lambda_1$ has the following series expansion:
 $$\lambda_1 = \sqrt{a_0} + c_1 + \frac{c_2}{\sqrt{a_0}} + \frac{c_3}{a_0} + \frac{c_4}{a_0^{\frac{3}{2}}} + O\left(a_0^{-2}\right).$$
 Here
\begin{align*}
c_1 &= \frac{a_1}{2a_0}\\
c_2 &= -\frac38 \left(\frac{a_1}{a_0}\right)^2 + \frac12 \frac{a_2}{a_0},
\\
c_3 &= \frac{a_1^3}{2a_0^3} -\frac{a_1a_2}{a_0^2} + \frac{a_3}{2a_0}\\
c_4 &=-\frac{105}{128} \left(\frac{a_1}{a_0}\right)^4 +\frac{35}{16} \left(\frac{a_1}{a_0}\right)^2\frac{a_2}{a_0}
-\frac{5}{8}\left(\frac{a_2}{a_0}\right)^2 -\frac{5}{4}\frac{a_1}{a_0}\frac{a_3}{a_0} +\frac{1}{2} \frac{a_4}{a_0}.
\end{align*}
 \end{lemma}

 The proof of Lemma~\ref{lem:series} is justified by Lemma 17 in~\cite{LLLW2022} and is similar to calculations given in~\cite{LLLW2022} and~\cite[Lemma 9]{LLW2022}. The following lemma compares the largest roots. 

\begin{lemma}\label{lem:compare}
   Let $f(\lambda)$ and $g(\lambda)$ be two decreasing functions on an interval $I$. Suppose
   $\lambda^2=f(\lambda)$ has a unique positive root $\lambda_f$ in $I$ 
   and $\lambda^2=g(\lambda)$ has a unique root $\lambda_g$ in $I$.
   If $f(\lambda)>g(\lambda)$ on I, then $\lambda_f>\lambda_g$.
\end{lemma}
\begin{proof}
    Assume towards a contradiction that $f(\lambda) > g(\lambda)$ on $I$ and $\lambda_f \le \lambda_g$. We observe
    \[
    0 = f(\lambda_f) - \lambda_f^2 > g(\lambda_f) - \lambda_f^2 \geq g(\lambda_g) - \lambda_f^2 \geq g(\lambda_g) - \lambda_g^2 = 0.
    \]
\end{proof}

The basic tools are the same for the cases when $n$ is even and $n$ is odd, but there are various technicalities which make it appropriate to deal with these two cases separately.   

\subsection{\texorpdfstring{When $n$ is even}{}}

\begin{proof}[Proof of Theorem~\ref{thm:t2}]

Let $n=2q$ and $G$ be the graph which achieves the maximum $\lambda_2$ among all connected outerplanar graphs on $n$ vertices. We would like to show that $G$ is isomorphic to $G_0=(K_1\vee P_{q-1})\!\!-\!\!(K_1\vee P_{q-1})$.
By equation~\eqref{eqn:lmb2eqn} in Lemma~\ref{lem:lamda2even}, $\lambda_2(G_0)$ satisfies the equation
$\lambda^2=f(\lambda)$, where
   \begin{equation*}
        f(\lambda) = (q-1) + \frac{2q-5}{\lambda} + \frac{4q-11}{\lambda^2} + \frac{8q-28}{\lambda^3} + \frac{16q-59}{\lambda^4} + \frac{32q-136}{\lambda^5} + O\left(\frac{q}{\lambda^6}\right).
    \end{equation*}
For $i\ge 0$, we define $a_i(G_0)$ to be the coefficient of the $\lambda^{-i}$ term in $f(\lambda)$, so that
\[a_i(G_0) = [\lambda^{-i}]f(\lambda).\]
Let $g(\lambda)=\sum_{i=0}^\infty \frac{a_i}{\lambda^i}$, where the $a_is$ are defined as in equation \eqref{eqn:aieqn} for the extremal graph $G$. Let $I=(\sqrt{n/2}+1 - \frac{d_1}{\sqrt{n}},
\sqrt{n/2}+1 + \frac{d_2}{\sqrt{n}})$, where $d_1$ and $d_2$ are sufficiently large constants. By Theorem~\ref{thm:t1} and Lemma~\ref{lem:lamda2even}, both of the
equations $\lambda^2=f(\lambda)$ and $\lambda^2=g(\lambda)$ have a unique root in $I$. We apply Lemma~\ref{lem:compare} to compare their roots $\lambda_f$ and $\lambda_g$.

We first give coarse upper bounds on $a_i$.
Since $G$ is an outerplanar graph, for $i=1,2$,
$G[{N(u_i)}]$ is a linear forest, which has at most $d_{u_i}-1$ edges.
Since $G$ is connected, there is at least one crossing edge connecting the two parts by Lemma~\ref{lem:u1u2notedge}. 
Thus, by \eqref{eqn:aiineq}
\[a_1\leq (d_{u_1}-1) + (d_{u_2}-1) -1 \leq 2q-2 +|N(u_1)\cap N(u_2)| -3 \leq 2q-3.\]

Since $G$ is outerplanar, the number of crossing edges between $N(u_1)$ and  $N(u_2)$ is at most $3$. These edges can only contribute $O(1)$ to $a_i$ for any fixed $i$. If we delete these crossing edges from $P$, we get a linear forest. They can contribute at most $2^iq$ to $a_i$.
Thus, for $i\geq 2$,  we have
\[a_i\leq 2^iq + O(1).\]

\noindent{\bf Claim 1:} $N(u_1) \cup N(u_2)$ forms a partition of $V(P).$ Furthermore, $d_{u_1} = d_{u_2} = q-1.$\vspace{0.5cm}

By \eqref{a_0ub},
\[2a_0\le d_{u_1}+ d_{u_2}- 2|N(u_1)\cap N(u_2)|= |N(u_1)\Delta N(u_2)|\leq 2q-2.\]
Equality holds if and only if $N(u_1)\cup N(u_2)$ forms a partition of $V(P)$. Suppose, towards a contradiction, that $N(u_1)\cup N(u_2)$ does not form a partition of $V(P)$. Then $2a_0 \leq 2q-3 $, and for any $\lambda\in I$, we have
\begin{align*}
    g(\lambda) &= a_0 + \frac{a_1}{\lambda} + \frac{a_2}{\lambda^2} + O\left(\frac{q}{\lambda^3}\right)\\
    &\leq \frac{1}{2}(2q-3) + \frac{2q-3}{\lambda} + \frac{4q+O(1)}{\lambda^2} + O\left(\frac{q}{\lambda^3}\right)\\
    &< (q-1) + \frac{2q-5}{\lambda} + \frac{4q-11}{\lambda^2} + O\left(\frac{q}{\lambda^3}\right)\\
    &=f(\lambda).
\end{align*}

By Lemma \ref{lem:compare}, we have $\lambda_{2,max}=\lambda_g<\lambda_f=\lambda_2(G_0),$ a contradiction. Hence, $N(u_1)\cup N(u_2)$ forms a partition of $V(P)$. By Lemma \ref{lem:degrees}, we have $d_{u_1} = d_{u_2} = q-1 $.\vspace{0.5cm}

\noindent{\bf Claim 2:} For $i=1,2$, $G[{N(u_i)}]$ is a path of length $q-2$. There is exactly one edge in 
$E(N(u_1), N(u_2))$.\vspace{0.5cm}

Otherwise, by \eqref{eqn:aiineq} we have
\[a_1< (q-2) + (q-2) -1 = 2q-5.\]
For any  $\lambda\in I$, we have
\begin{align*}
    g(\lambda) &= a_0 + \frac{a_1}{\lambda} + \frac{a_2}{\lambda^2} + \frac{a_3}{\lambda^3} +
   O\left(\frac{q}{\lambda^4}\right)\\
    &\leq q-1 + \frac{2q-6}{\lambda} + \frac{4q+O(1)}{\lambda^2} + \frac{8q+O(1)}{\lambda^3} + O\left(\frac{q}{\lambda^4}\right)\\
    &< (q-1) + \frac{2q-5}{\lambda} + \frac{4q-11}{\lambda^2} + \frac{8q-28}{\lambda^3} + O\left(\frac{q}{\lambda^4}\right)\\
    &=f(\lambda).
\end{align*}
By Lemma \ref{lem:compare}, we have $\lambda_{2,max}=\lambda_g<\lambda_f=\lambda_2(G_0),$ a contradiction. Hence, Claim 2 is proven. \vspace{0.5cm}

\noindent{\bf Claim 3}: The unique edge in $E(N(u_1), N(u_2))$ connects the ending vertices of $G[N(u_1)]$ and $G[N(u_2)]$.\vspace{0.5cm}

Denote this edge by $uv$. The contribution of $uv$ in $a_2$ is given by
$$1 -(d_u-1+d_v-1)< 3 - (2 + 2) = -1.$$
If either $u$ or $v$ are not ending vertices, we have
\[a_2 \leq a_2(G_0)-1 \leq 4q-12.\]
For any  $\lambda\in I$, we have
\begin{align*}
    g(\lambda) &= a_0 + \frac{a_1}{\lambda} + \frac{a_2}{\lambda^2} + \frac{a_3}{\lambda^3} + \frac{a_4}{\lambda^4}+ O\left(\frac{q}{\lambda^5}\right)\\
    &\leq q-1 + \frac{2q-5}{\lambda} + \frac{4q-12}{\lambda^2} + \frac{8q+O(1)}{\lambda^3} + \frac{16q+O(1)}{\lambda^4}+O\left(\frac{q}{\lambda^5}\right)\\
    &< (q-1) + \frac{2q-5}{\lambda} + \frac{4q-11}{\lambda^2} + \frac{8q-28}{\lambda^3} + \frac{16q-59}{\lambda^4} + O\left(\frac{q}{\lambda^5}\right)\\
    &=f(\lambda).
\end{align*}
By Lemma \ref{lem:compare}, we have $\lambda_{2,max}=\lambda_g<\lambda_f=\lambda_2(G_0),$ a contradiction. Thus, Claim 3 is proven, $G$ is isomorphic to $G_0$ and the proof of Theorem \ref{thm:t2} is complete.
\end{proof}



We now give the proof of Corollary~\ref{cor:c1}, which determines the maximum size of $\lambda_2$ over all outerplanar graphs on $n=2q$ vertices with $n$ sufficiently large. 

\begin{proof}[Proof of Corollary~\ref{cor:c1}]
Theorem~\ref{thm:t2} determines the unique outerplanar graph with maximum $\lambda_2$ over connected outerplanar graphs, so we only need to consider disconnected outerplanar graphs. Let $G$ be the disconnected outerplanar graph on $n=2q$ vertices with maximum $\lambda_2$ and suppose that $G$ has connected components $G_1, G_2\ldots, G_m$, labelled so that $\lambda_1(G_1) \ge \lambda_1(G_2) \ge \cdots \ge \lambda_1(G_m)$. Note that $\lambda_2(G_i) <  \lambda_2((K_1\vee P_{q-1})\!\!-\!\!(K_1\vee P_{q-1}))$ for each $i$ and for $n$ sufficiently large, so we only need to consider the case where $\lambda_2(G) = \lambda_1(G_2)$. By the result of Tait and Tobin~\cite{TT2017} on the maximum spectral radius of outerplanar graphs, this is achieved when $G_2 = K_1 \vee P_{q-1}$. Using the series expansions in Lemma~\ref{lem:specradfangph} and Lemma~\ref{lem:lamda2even}, we see that \[\lambda_1(K_1 \vee P_{q-1}) = \sqrt{q-1} + 1 + \frac{1}{2\sqrt{q-1}} - \frac{1}{q-1} - \frac{1}{8(q-1)^{3/2}} - \frac{7}{16(q-1)^{5/2}} + O\left(\frac{1}{(q-1)^{3}}\right),\] 
while \[ \lambda_2((K_1\vee P_{q-1})\!\!-\!\!(K_1\vee P_{q-1})) = \sqrt{q-1} + 1 + \frac{1}{2\sqrt{q-1}} - \frac{3}{2(q-1)} + \frac{7}{8(q-1)^{\frac32}} - \frac{2}{(q-1)^2} + O\left(\frac{1}{(q-1)^{5/2}}\right),\] so for sufficiently large $n$, 
\[\lambda_1(K_1\vee P_{q-1}) > \lambda_2((K_1\vee P_{q-1})\!\!-\!\!(K_1\vee P_{q-1})),\]
completing the proof. 
\end{proof}

\subsection{\texorpdfstring{When $n$ is odd}{}}

\begin{proof}[Proof of Theorem~\ref{thm:t3}]

Let $n=2q+1$ and $G$ be the graph that reaches the maximum $\lambda_2$ among all connected outerplanar graphs on $n$ vertices. Let $G_0$ be a graph obtained by adding a new vertex connecting two copies of $(K_1\vee P_{q-1})$.
By Cauchy's Interlacing Theorem and Lemma \ref{lem:specradfangph}, $\lambda_2(G_0)$ satisfies the equation
$\lambda^2=f(\lambda)$, where
   \begin{equation*}
        f(\lambda) = (q-1) + \frac{2q-4}{\lambda} + \frac{4q-10}{\lambda^2} + \frac{8q-24}{\lambda^3} + \frac{16q-54}{\lambda^4} + \frac{32q-120}{\lambda^5} + O\left(\frac{q}{\lambda^6}\right).
    \end{equation*}
For $i\ge 0$, we define
\[a_i(G_0) = [\lambda^{-i}]f(\lambda).\]
Let $g(\lambda)=\sum_{i=0}^\infty \frac{a_i}{\lambda^i}$, where the $a_i$s are defined as in \eqref{eqn:aieqn} for the extremal graph $G$. Let \[I=\left(\sqrt{n/2}+1 - \frac{d_1}{\sqrt{n}},
\sqrt{n/2}+1 + \frac{d_2}{\sqrt{n}}\right),\] where $d_1$ and $d_2$ are sufficiently large constants. By Theorem \ref{thm:t1} and Lemma \ref{lem:specradfangph}, both of the
equations $\lambda^2=f(\lambda)$ and $\lambda^2=g(\lambda)$ have a unique root in $I$. We apply Lemma \ref{lem:compare} to compare their roots $\lambda_f$ and $\lambda_g$.

By a similar argument as for the even case, for $i\geq 2$,  we have
\[a_i\leq 2^iq + O(1).\]
{\bf Claim 1:}  One of the following four cases must occur:
\begin{enumerate}
    \item $d_{u_1}=d_{u_2}=q-1$, and $N(u_1)\cap N(u_2)=\emptyset$.
    \item $d_{u_1}=d_{u_2}=q$, and $|N(u_1)\cap N(u_2)|=1$.
    \item $d_{u_1}=q$, $d_{u_2}=q-1$, and $N(u_1)\cap N(u_2)=\emptyset$.
    \item $d_{u_1}=q-1$, $d_{u_2}=q$, and $N(u_1)\cap N(u_2)=\emptyset$.
\end{enumerate}

Otherwise, we show $a_0\leq \frac{1}{2}(2q-3)$. Indeed, by \eqref{eqn:aiineq}
\[2a_0\le d_{u_1}+ d_{u_2}- 2|N(u_1)\cap N(u_2)|= |N(u_1)\Delta N(u_2)|\leq 2q-1.\]
If $|N(u_1)\Delta N(u_2)|=2q-1$, we get case 3 or case 4. If $|N(u_1)\Delta N(u_2)|=2q-2$, we get case 1 or case 2. If none of these cases occur, we get
\[a_0\leq \frac{2q-3}{2}. \]
For any  $\lambda\in I$, we have
\begin{align*}
    g(\lambda) &= a_0 + \frac{a_1}{\lambda} + \frac{a_2}{\lambda^2} + O\left(\frac{q}{\lambda^3}\right)\\
    &\leq \frac{1}{2}(2q-3) + \frac{2q+O(1)}{\lambda} + \frac{4q+O(1)}{\lambda^2} + O\left(\frac{q}{\lambda^3}\right)\\
    &< (q-1) + \frac{2q-4}{\lambda} + \frac{4q-10}{\lambda^2} + O\left(\frac{q}{\lambda^3}\right)\\
    &=f(\lambda).
\end{align*}
By Lemma \ref{lem:compare}, we have $\lambda_{2,max}=\lambda_g<\lambda_f=\lambda_2(G_0),$ a contradiction and Claim 1 is proven.\vspace{0.5cm}

\noindent\textit{\textbf{Case 1:}} $d_{u_1}=d_{u_2}=q-1$, and $N(u_1)\cap N(u_2)=\emptyset$.
Let $u$ be the unique vertex not in $N[u_1] \cup N[u_2]$.\vspace{0.5cm} {\bf Claim 2:} For $i = 1, 2$, $G[{N (u_i)}]$ is a path of length $q-2$. Moreover, $u$ is a cut vertex of $G$.\vspace{0.5cm}

Observe $u$ has zero contribution to $a_1$, because there are no walks of length $1$ with endpoints in $N(u_1)$ or $N(u_2)$ which include $u$. If $u$ is not a cut vertex, then,
there is at least an edge in $E(N(u_1), N(u_2))$, which contributes $-1$ to $a_1$. If either $G[{N (u_1)}]$ or $G[{N (u_2)}]$ is not $P_{q-1}$, then it contributes one less in $a_1$. Therefore, if Claim 2 does not hold, we have
\[a_1\leq (q-2) + (q-2) -1 = 2q-5.\]
For any  $\lambda\in I$, we have
\begin{align*}
    g(\lambda) &= a_0 + \frac{a_1}{\lambda} + \frac{a_2}{\lambda^2} + \frac{a_3}{\lambda^3} + O\left(\frac{q}{\lambda^4}\right)\\
    &\leq q-1 + \frac{2q-5}{\lambda} + \frac{4q+O(1)}{\lambda^2} + \frac{8q+O(1)}{\lambda^3} + O\left(\frac{q}{\lambda^4}\right)\\
    &< (q-1) + \frac{2q-4}{\lambda} + \frac{4q-10}{\lambda^2} + \frac{8q-24}{\lambda^3} + O\left(\frac{q}{\lambda^4}\right)\\
    &=f(\lambda).
\end{align*}
By Lemma \ref{lem:compare}, we have $\lambda_{2,max}=\lambda_g<\lambda_f=\lambda_2(G_0),$ a contradiction and Claim 2 is proven.\vspace{0.5cm}\\
\noindent
\textit{\textbf{Case 2:}} $d_{u_1}=d_{u_2}=q$, and $|N(u_1)\cap N(u_2)|=1$. Let $u$ be the unique vertex in $N(u_1)\cap N(u_2)$.\vspace{0.5cm}\\
\noindent
{\bf Claim 3}: For $i = 1, 2$, $G[{N (u_i)}] \setminus \{u\}$ is a path of length $q-2$. Moreover, $u$ is a cut vertex of $G$.\vspace{0.5cm}

Observe $u$ has zero contribution to $a_1$. If $u$ is not a cut vertex, then,
there is at least an edge in $E(N(u_1)\setminus\{u\} , N(u_2)\setminus\{u\})$, which contributes $-1$ to $a_1$. If either $G[N (u_1)] \setminus \{u\}$ or $G[N (u_2)] \setminus \{u\}$ is not $P_{q-1}$, then it contributes one less in $a_1$. Therefore, if Claim 2 does not hold, we have
\[a_1\leq (q-2) + (q-2) -1 = 2q-5.\]
For any  $\lambda\in I$, we have
\begin{align*}
    g(\lambda) &= a_0 + \frac{a_1}{\lambda} + \frac{a_2}{\lambda^2} + \frac{a_3}{\lambda^3} + O\left(\frac{q}{\lambda^4}\right)\\
    &\leq q-1 + \frac{2q-5}{\lambda} + \frac{4q+O(1)}{\lambda^2} + \frac{8q+O(1)}{\lambda^3} + O\left(\frac{q}{\lambda^4}\right)\\
    &< (q-1) + \frac{2q-4}{\lambda} + \frac{4q-10}{\lambda^2} + \frac{8q-24}{\lambda^3} + O\left(\frac{q}{\lambda^4}\right)\\
    &=f(\lambda).
\end{align*}
By Lemma \ref{lem:compare}, we have $\lambda_{2,max}=\lambda_g<\lambda_f=\lambda_2(G_0),$ a contradiction and Claim 3 is proven.\vspace{0.5cm}

\noindent
\textit{\textbf{Case 3:}} $d_{u_1}=q$, $d_{u_2}=q-1$, and $N(u_1)\cap N(u_2)=\emptyset$.\vspace{0.5cm}

From equation \eqref{eq:gamma1} and \eqref{eq:gamma2}, $\lambda_2$ satisfies two equations.
\begin{align}
\lambda^2&= q + \sum_{i=1}^{\infty}\frac{\sum_{(u,v) \in N(u_1)^2} w_i(u,v)
+ \sum_{(u,v) \in N(u_1) \times N(u_2)} w_i(u,v) \frac{x_{u_2}}{x_{u_1}}}{\lambda^i}, \nonumber\\
\lambda^2&= q-1 + \sum_{i=1}^{\infty}\frac{\sum_{(u,v) \in N(u_2)^2} w_i(u,v)
+ \sum_{(u,v) \in N(u_1) \times N(u_2)} w_i(u,v) \frac{x_{u_1}}{x_{u_2}}}{\lambda^i}.  
\label{eq:gamma2exp}
\end{align}
Taking the difference and solving for $\frac{x_{u_2}}{x_{u_1}}- \frac{x_{u_1}}{x_{u_2}}$, we have
\begin{align*}
   \frac{x_{u_2}}{x_{u_1}} - \frac{x_{u_1}}{x_{u_2}}
&= -\frac{1+ \sum_{i=1}^{\infty}\frac{\sum_{(u,v) \in N(u_1)^2} w_i(u,v) - \sum_{(u,v) \in N(u_2)^2} w_i(u,v)}{\lambda^i}} {\sum_{i=1}^{\infty} \frac{\sum_{(u,v) \in N(u_1) \times N(u_2)} w_i(u,v)}{\lambda_i} }  \\
&= - \frac{(1+O(\frac{1}{\lambda}))}{ \frac{|E(N(u_1), N(u_2))|}{\lambda}}\\
&= -\left(1+O\left(\frac{1}{\lambda}\right)\right) \frac{\lambda}{|E(N(u_1), N(u_2))|}.
\end{align*}
Solving for $\frac{x_{u_2}}{x_{u_1}}$, we get
 \begin{equation}\label{eqn:xu2xu1}
 \frac{x_{u_2}}{x_{u_1}} =-\left(1+O\left(\frac{1}{\lambda}\right)\right) \frac{\lambda}{|E(N(u_1), N(u_2))|}.
 \end{equation}\vspace{0.5cm}\\
\noindent{\bf Claim 4}: $G[N(u_2)]$ forms a path of length $q-2$.\vspace{0.5cm}

Otherwise, we have
\[ \sum_{(u,v) \in N(u_2)^2} w_1(u,v) =2 |E(G[{N(u_2)}])|\leq 2(q-3). \]
Let $g_2(\lambda)=\sum_{i=0}^\infty \lambda_2^{-i}\gamma_2^TA_P^i\beta$ as in equation \eqref{eq:gamma2}. 
Since $\frac{x_{u_1}}{x_{u_2}}<0$, 
using equation \eqref{eq:gamma2exp}, 

\begin{align*}
    g(\lambda)
    &\leq q-1 + \frac{2q-6}{\lambda} + \frac{4q+O(1)}{\lambda^2} + \frac{8q+O(1)}{\lambda^3} + O\left(\frac{q}{\lambda^4}\right)\\
    &< (q-1) + \frac{2q-4}{\lambda} + \frac{4q-10}{\lambda^2} + \frac{8q-24}{\lambda^3} + O\left(\frac{q}{\lambda^4}\right)\\
    & = f(\lambda). 
\end{align*}
By Lemma \ref{lem:compare}, we have $\lambda_{2,max}=\lambda_g<\lambda_f=\lambda_2(G_0),$ a contradiction and Claim 4 is proven.\vspace{0.5cm}

\noindent{\bf Claim 5}: $|E(N(u_1), N(u_2))| \in \{1, 2\}$. If $|E(N(u_1), N(u_2))|=2$,
then the two edges in $E(N(u_1), N(u_2))$ share a common vertex in $N(u_1)$.\vspace{0.5cm}

By Claim 4, the induced graph on $N(u_2)$ is a path $P_{q-1}$. Let $c= |E(N(u_1), N(u_2))|$. Let us estimate the contribution of these crossing edges in $\sum_{u,v\in N(u_2)} w_2(u,v)$. Since $G$ is a connected outerplanar graph, we have
$c\in\{1,2,3\}$.  
Let $\eta$ be the contribution of the crossing edges in $\sum_{u,v\in N(u_2)} w_2(u,v)$. We have
\[\eta=\begin{cases}
    1 & \mbox{ if } c=1,\\
    4 & \mbox{ if } c=2 \mbox{ and two crossing edges share a common vertex in } N(u_1), \\
    2 & \mbox { if } c=2 \mbox { and two crossing edges do not share a common vertex in } N(u_1),\\
    5 & \mbox { if } c=3.
\end{cases}\]

When $c=3$, the three crossing edges must form a shape of $N$, thus the contribution is $3+2=5$. Note that the contribution of $P_{q-1}$ to $\sum_{u,v\in N(u_2)} w_2(u,v)$ is
$$2(q-1)-2 + 2(q-3)=4q-10.$$

For any  $\lambda\in I$, using \eqref{eqn:xu2xu1} and \eqref{eq:gamma2} we have
\begin{align*}
    g_2(\lambda) &=\sum_{i=0}^\infty \lambda_2^{-i}\gamma_2^TA_P^i\beta\\
    &\leq q-1 + \frac{2q-4 + c \frac{x_{u_1}}{x_{u_2}}}{\lambda} + \frac{4q-10
    + \eta +O(1)\cdot \frac{x_{u_1}}{x_{u_2}} }{\lambda^2} + \frac{8q+O(1)}{\lambda^3} + O\left(\frac{q}{\lambda^4}\right)\\
     &\leq q-1 + \frac{2q-4 - \frac{c^2}{\lambda}}{\lambda} + \frac{4q-10
    + \eta} {\lambda^2} + \frac{8q+O(1)}{\lambda^3} + O\left(\frac{q}{\lambda^4}\right)\\
    &= (q-1) + \frac{2q-4}{\lambda} + \frac{4q-10 -c^2+\eta}{\lambda^2} + \frac{8q+O(1)}{\lambda^3} + O\left(\frac{q}{\lambda^4}\right)\\
     &< (q-1) + \frac{2q-4}{\lambda} + \frac{4q-10}{\lambda^2} + \frac{8q-24}{\lambda^3} + O\left(\frac{q}{\lambda^4}\right)\\
    &=f(\lambda).
\end{align*}

If Claim 5 fails, we have $\eta-c^2\leq -2$. Thus, the last inequality holds.
By Lemma \ref{lem:compare}, we have $\lambda_{2,max}=\lambda_g<\lambda_f=\lambda_2(G_0),$ a contradiction and Claim 5 is proven. Therefore, there exists a cut vertex of $G$ in $N(u_1)$ which we label $u.$\vspace{0.5cm}

\noindent{\bf Claim 6:} $G\setminus \{u\}= 2 K_1\vee P_{q-1}$.\vspace{0.5cm}

By Claim $4$, one of the components in $G\setminus \{u\}$ is $K_1\vee P_{q-1}$.
Call the other component $H$. If $H\not= K_1\vee P_{q-1}$, by Tait and Tobin's result, we have 
\[\lambda_1(H) <\lambda_1 (K_1\vee P_{q-1}).\]
Thus, we have
\begin{align*}
    \lambda_1(G\setminus \{u\}) & = \lambda_1(K_1\vee P_{q-1}),\\
    \lambda_2(G\setminus \{u\}) & = \lambda_1(H).
\end{align*}
Since two eigenvalues are not equal,
by Cauchy's Interlacing Theorem, we have
\[ \lambda_2(G_0) = \lambda_1(K_1\vee P_{q-1})=\lambda_1(G\setminus \{u\}) \ge \lambda_2(G)\ge\lambda_2(G\setminus \{u\})=\lambda_1(H),\]
and the middle inequalities cannot be equalities simultaneously. Thus, $\lambda_2(G)$ is not maximal if $H\neq K_1\vee P_{q-1}$.\vspace{0.5cm}

\noindent\textit{\textbf{Case 4:}} $d_{u_1}=q-1$, $d_{u_2}=q$, and $N(u_1)\cap N(u_2)=\emptyset$.
This is symmetric to Case 3.\vspace{0.5cm}

Therefore, $G$ always contains a cut vertex $u$ such that $G\setminus \{u\}$ is two disjoint copies of $K_1\vee P_{q-1}$.
The proof of Theorem \ref{thm:t3} is finished.
\end{proof}

\section{Maximum \texorpdfstring{$\lambda_2$}{} for 2-connected outerplanar graphs}
We first compute the series expansion of $\lambda_2$ for the conjectured extremal graph $(K_1\vee P_{\lfloor \frac{n}{2}\rfloor-1})\diamond(K_1\vee P_{\lceil \frac{n}{2}\rceil-1})$. 
\begin{lemma}
Let $G= (K_1\vee P_{\lfloor \frac{n}{2}\rfloor-1})\diamond(K_1\vee P_{\lceil \frac{n}{2}\rceil-1})$. Then:
\begin{enumerate}
    \item For $n=2q$, $\lambda_2(G)$ satisfies the following equation.
    \begin{equation}
    \lambda^2 = (q-1) + \frac{2q-6}{\lambda} + \frac{4q-14}{\lambda^2} + \frac{8q-32}{\lambda^3} + \frac{16q-70}{\lambda^4}+\frac{32q-152}{\lambda^5} 
    +\frac{64q-324}{\lambda^6} +O\left(\frac{q}{\lambda^7}\right).
    \label{eq:2conneven}
\end{equation}
In particular, we have
\begin{equation*}
    \lambda_2 = \sqrt{q-1} + 1 + \frac{1}{2\sqrt{q-1}} - \frac{2}{q-1} + \frac{7}{8(q-1)^{3/2}}  + \frac{113}{16(q-1)^{5/2}} + O\left(\frac{1}{(q-1)^{3}}\right).
\end{equation*}

 \item For $n=2q+1$, $\lambda_2(G)$ satisfies the following equation.
    \begin{equation}
    \lambda^2 = 
    (q-1) + \frac{2q-4}{\lambda} + \frac{4q-12}{\lambda^2} + \frac{8q-32}{\lambda^3} + \frac{16q-64}{\lambda^4}+\frac{32q-112}{\lambda^5} +\frac{64q-232}{\lambda^6} +O\left(\frac{q}{\lambda^7}\right).
     \label{eq:2connodd}
\end{equation}
In particular, we have
\begin{equation*}
    \lambda_2 = \sqrt{q-1} + 1 + \frac{1}{2\sqrt{q-1}} - \frac{1}{q-1} - \frac{9}{8(q-1)^{3/2}} -\frac{87}{16(q-1)^{5/2}}+ O\left(\frac{1}{(q-1)^{3}}\right)
\end{equation*}
\end{enumerate}
\end{lemma}

\begin{proof}
   Let $u_1$ be the vertex of the largest degree and $u_2$ be the vertex of the second largest degree in $G$. 
    Let $P$ be the induced subgraph obtained from $G$ after deleting $u_1$ and $u_2$. Let $\bx$ be the eigenvector of $G$ corresponding to $\lambda_2$. Note that $\bx$ is a simple eigenvector because by Cauchy's Interlacing Theorem $\lambda_2$ is a simple eigenvalue, as $\lambda_2(G) = \Omega(\sqrt{n})$, while $\lambda_3(G) \le \lambda_1(G\setminus\{u_1, u_2\}) \le \Delta(G\setminus\{u_1, u_2\}) \le 4$. 
Let $\beta_1$, $\beta_2$, $\gamma_1$, $\gamma_2$, $\beta$, $\gamma$ defined as in equations \eqref{eq:beta1def}-\eqref{eq:gamma}. 
equation \eqref{eq:gamma1} can be re-written as
\[\lambda^2= F_1(\lambda) + D(\lambda) \frac{x_{u_2}}{x_{u_1}}.\]
Equation \eqref{eq:gamma2} can be re-written as
\[\lambda^2= F_2(\lambda) + D(\lambda) \frac{x_{u_1}}{x_{u_2}}.\]
Here
\begin{align*}
F_1(\lambda)&=\sum_{i=0}^{\infty}\frac{\sum_{(u,v) \in N(u_1)^2} w_i(u,v)}{\lambda^i}, \\
F_2(\lambda)&=\sum_{i=0}^{\infty}\frac{\sum_{(u,v) \in N(u_2)^2} w_i(u,v)}{\lambda^i}, \\
D(\lambda)&=\sum_{i=1}^{\infty}\frac{\sum_{(u,v) \in N(u_1) \times N(u_2)} w_i(u,v)}{\lambda^i}.
\end{align*}
Cancelling $\frac{x_{u_2}}{x_{u_1}}$, we get
\[ (\lambda^2-F_1(\lambda))
(\lambda^2-F_2(\lambda)) = D(\lambda)^2,
\]
which implies
\[\lambda^2 =\frac{1}{2}
\left( (F_1(\lambda) +F_2(\lambda)) 
- \sqrt{ (F_1(\lambda) -F_2(\lambda))^2 + 4 D(\lambda)^2}
\right).
\]
Using SageMath, we can calculate
\begin{align*}
    F_2(\lambda) &=  (q-1) + \frac{2q-4}{\lambda} + \frac{4q-8}{\lambda^2} + \frac{8q-16}{\lambda^3} + \frac{16q-28}{\lambda^4}+\frac{32q-48}{\lambda^5}
    +\frac{64q-64}{\lambda^6}+O\left(\frac{q}{\lambda^7}\right).\\
    D(\lambda)&=  \frac{2}{\lambda} + \frac{6}{\lambda^2} + \frac{16}{\lambda^3} + \frac{42}{\lambda^4}+\frac{104}{\lambda^5} 
    +\frac{260}{\lambda^6}+O\left(\frac{q}{\lambda^7}\right).
\end{align*}
When $n=2q$, we have $F_1(\lambda)=F_2(\lambda)$ and
\[\lambda^2=F_2(\lambda)-D(\lambda).\]
This implies equation \eqref{eq:2conneven}. Using SageMath, we can calculate
\begin{equation*}
    \lambda_2 = \sqrt{q-1} + 1 + \frac{1}{2\sqrt{q-1}} - \frac{2}{q-1} + \frac{7}{8(q-1)^{3/2}}  + \frac{113}{16(q-1)^{5/2}} + O\left(\frac{1}{(q-1)^{3}}\right).
\end{equation*}
When $n=2q+1$, we have
\begin{align*}
    F_1(\lambda)&=  q + \frac{2q-2}{\lambda} + \frac{4q-4}{\lambda^2} + \frac{8q-8}{\lambda^3} + \frac{16q-12}{\lambda^4}+\frac{32q-16}{\lambda^5}  +\frac{64q}{\lambda^6}+O\left(\frac{q}{\lambda^7}\right).
\end{align*}
Therefore,
\begin{align*}
    \lambda^2 &=\frac{1}{2}
\left( (F_1(\lambda) +F_2(\lambda)) 
- \sqrt{ (F_1(\lambda) -F_2(\lambda))^2 + 4 D(\lambda)^2}
\right)\\
&= \frac{1}{2}
\left(2q-1 + \frac{4q-6}{\lambda} + \frac{8q-12}{\lambda^2} + \frac{16q-24}{\lambda^3} + \frac{32q-40}{\lambda^4}+\frac{64q-64}{\lambda^5} +\frac{128q-64}{\lambda^6} 
 +O\left(\frac{q}{\lambda^7}\right)\right.\\
&\hspace*{4mm}
\left.
-\sqrt{
\left(1+ \frac{2}{\lambda} + \frac{4}{\lambda^2} + \frac{8}{\lambda^3} + \frac{16}{\lambda^4}+\frac{32}{\lambda^5} +\frac{64}{\lambda^6}+O\left(\frac{1}{\lambda^7}\right)\right)^2
+ 4\left(\frac{2}{\lambda} + \frac{6}{\lambda^2} + \frac{16}{\lambda^3} + \frac{42}{\lambda^4}+\frac{104}{\lambda^5} +\frac{260}{\lambda^6} +O\left(\frac{q}{\lambda^7}\right)\right)^2
}
\right) \\
&= (q-1) + \frac{2q-4}{\lambda} + \frac{4q-12}{\lambda^2} + \frac{8q-32}{\lambda^3} + \frac{16q-64}{\lambda^4}+\frac{32q-112}{\lambda^5} +\frac{64q-232}{\lambda^6} +O\left(\frac{q}{\lambda^7}\right).
\end{align*}
Using SageMath, we can calculate the series expansion of $\lambda_2$.

\begin{equation*}
    \lambda_2 = \sqrt{q-1} + 1 + \frac{1}{2\sqrt{q-1}} - \frac{1}{q-1} - \frac{9}{8(q-1)^{3/2}} -\frac{87}{16(q-1)^{5/2}}+ O\left(\frac{1}{(q-1)^{3}}\right).
\end{equation*}
\end{proof}

\begin{proof}[Proof of Theorem \ref{thm:t4}]

Let $G$ be the graph which achieves the maximum $\lambda_2$ among all 2-connected outerplanar graphs on $n$ vertices. Let $G_0$ be the conjectured extremal 2-connected outerplanar graph on $n$ vertices.
We follow the proofs of Theorems \ref{thm:t2} and \ref{thm:t3}. Assume $\lambda_2(G)$ satisfies the equation
\[\lambda^2=g(\lambda),\]
where $g(\lambda) = \sum_{i=0}^{\infty}\frac{a_i}{\lambda^i}$ and $a_i$ is defined as in \eqref{eqn:aieqn} for $G$,
while $\lambda_2(G_0)$ satisfies the equation
\[\lambda^2=f(\lambda),\]
where $f(\lambda)$ is the right-hand side of \eqref{eq:2conneven} or \eqref{eq:2connodd} depending on the parity of $n$. In either case, we define
\[a_i(G_0) = [\lambda^{-i}]f(\lambda).\]
It is sufficient to compare $f(\lambda)$ and $g(\lambda)$ in a small interval
$I= \left(\sqrt{n/2}+1 - \frac{c_1}{\sqrt{n}},
\sqrt{n/2}+1 + \frac{c_2}{\sqrt{n}}\right)$.\vspace{0.5cm}

\noindent\textit{\textbf{Case $n=2q$:}}
We have
   \begin{equation*} \label{eq:t2f_even_2_conn}
        f(\lambda) = (q-1) + \frac{2q-6}{\lambda} + \frac{4q-14}{\lambda^2} + \frac{8q-32}{\lambda^3} + \frac{16q-70}{\lambda^4} + \frac{32q-152}{\lambda^5} + O\left(\frac{q}{\lambda^6}\right).
    \end{equation*}
Consider the following claims.\vspace{0.5cm}

\noindent{\bf Claim 1:} $d_{u_1}=d_{u_2}=q-1$ and $N(u_1)\cap N(u_2)=\emptyset$.\vspace{0.5cm}

The proof is identical to the proof of Claim 1 in Theorem \ref{thm:t2}.\vspace{0.5cm}

\noindent{\bf Claim 2:} For $i=1,2$, $G[{N(u_i)}]$ is a path of length $q-2$. There are exactly two edges in 
$E(N(u_1), N(u_2))$. In particular, the two edges form a matching.\vspace{0.5cm}

Otherwise, we have
\[a_1< (q-2) + (q-2) -2 = 2q-6.\]
For any  $\lambda\in I$, we have
\begin{align*}
    g(\lambda) &= a_0 + \frac{a_1}{\lambda} + \frac{a_2}{\lambda^2} + \frac{a_3}{\lambda^3} +
    O\left(\frac{q}{\lambda^4}\right)\\
    &\leq q-1 + \frac{2q-7}{\lambda} + \frac{4q+O(1)}{\lambda^2} + \frac{8q+O(1)}{\lambda^3} + O\left(\frac{q}{\lambda^4}\right)\\
    &< (q-1) + \frac{2q-6}{\lambda} + \frac{4q-14}{\lambda^2} + \frac{8q-32}{\lambda^3} + O\left(\frac{q}{\lambda^4}\right)\\
    &=f(\lambda).
\end{align*}
By Lemma \ref{lem:compare}, we have $\lambda_{2,max}=\lambda_g<\lambda_f=\lambda_2(G_0),$ a contradiction and Claim 2 is proven. Since $G$ is 2-connected, the two crossing edges must form a matching.\vspace{0.5cm}

\noindent{\bf Claim 3}: The ending vertices of the two edges in $E(N(u_1), N(u_2))$ have degrees
$1,1,2,2$ in $P$.\vspace{0.5cm}

Denote the two edges by $v_1w_1$ and $v_2w_2$.
The contribution of $v_iw_i$ in $a_2$ is given by
$$1 -\left(d^P_{v_i}+d^P_{w_i}\right).$$
We have
\[a_2 \leq 4q-8 -\left(d^P_{v_1}+d^P_{w_1}
+d^P_{v_2}+d^P_{w_2}\right).\]
Note that since $G$ is outerplanar, at most two vertices in $v_1, w_1, v_2, w_2$ are ending vertices. Thus, we have
\[a_2 \leq 4q-8 -\left(d^P_{v_1}+d^P_{w_1}
+d^P_{v_2}+d^P_{w_2}\right)\leq 4q-14,\]
with equality if and only if the
degrees $d^P_{v_1}, d^P_{w_1}, d^P_{v_2}, \text{and } d^P_{w_2}$
are $1,1,2, \text{and }2$, respectively. If Claim 3 fails, we have
\[a_2\leq 4q-15.\]
For any  $\lambda\in I$, we have
\begin{align*}
    g(\lambda) &= a_0 + \frac{a_1}{\lambda} + \frac{a_2}{\lambda^2} + \frac{a_3}{\lambda^3} + \frac{16q+O(1)}{\lambda^3} +
    O\left(\frac{q}{\lambda^4}\right)\\
    &\leq q-1 + \frac{2q-6}{\lambda} + \frac{4q-15}{\lambda^2} + \frac{8q+O(1)}{\lambda^3} +  \frac{16q+O(1)}{\lambda^4} +
    O\left(\frac{q}{\lambda^4}\right)\\
    &< (q-1) + \frac{2q-6}{\lambda} + \frac{4q-14}{\lambda^2} + \frac{8q-32}{\lambda^3} 
    + \frac{16q-70}{\lambda^4} + 
    O\left(\frac{q}{\lambda^4}\right)\\
    &=f(\lambda).
\end{align*}
By Lemma \ref{lem:compare}, we have $\lambda_{2,max}=\lambda_g<\lambda_f=\lambda_2(G_0),$ a contradiction and Claim 3 is proven.\vspace{0.5cm}

From Claims 1-3, there are only two graphs left. Call the other graph $G'_0$. Both $G_0$ and $G'_0$ have the same values on $a_0$, $a_1$, and $a_2$. They only differ by $1$ on $a_3$. The graph $G'_0$ contains one more negative $3$-walk, which is highlighted in red in Figure 
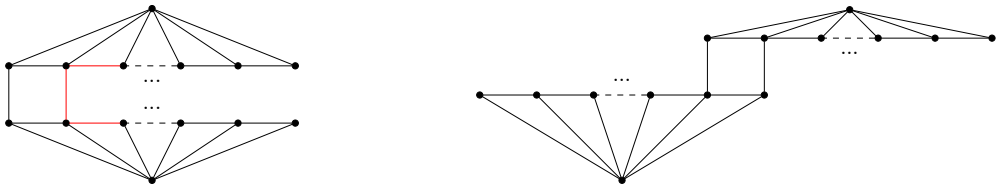
\begin{figure}[hb]
    \centering
 \resizebox{4cm}{!}{\begin{tikzpicture}[scale=1, Wvertex/.style={circle, draw=black, fill=white, scale=1}, bvertex/.style={circle, draw=black, fill=black, scale=0.3}]

\node [bvertex] (u0) at (-0.5, -1.5) {};
\node [bvertex] (u1) at (-3, -0.5) {};
\node [bvertex] (u2) at (-2, -0.5) {};
\node [bvertex] (u3) at (-1, -0.5) {};
\node [bvertex] (u4) at (0, -0.5) {};
\node [bvertex] (u5) at (1, -0.5) {};
\node [bvertex] (u6) at (2, -0.5) {};

\node [bvertex] (v0) at (-0.5, 1.5) {};
\node [bvertex] (v1) at (-3, 0.5) {};
\node [bvertex] (v2) at (-2, 0.5) {};
\node [bvertex] (v3) at (-1, 0.5) {};
\node [bvertex] (v4) at (0, 0.5) {};
\node [bvertex] (v5) at (1, 0.5) {};
\node [bvertex] (v6) at (2, 0.5) {};

\draw (u0) -- (u1);
\draw (u0) -- (u2);
\draw (u0) -- (u3);
\draw (u0) -- (u4);
\draw (u0) -- (u5);
\draw (u0) -- (u6);

\draw (u1) -- (u2);
\draw[color=red, ] (u3) -- (u2); 
\draw [dashed] (u4) -- (u3) node [midway, fill=white, above=3pt] {$...$};
\draw (u5) -- (u4);
\draw (u6) -- (u5);

\draw (v0) -- (v1);
\draw (v0) -- (v2);
\draw (v0) -- (v3);
\draw (v0) -- (v4);
\draw (v0) -- (v5);
\draw (v0) -- (v6);

\draw (v1) -- (v2);
\draw[color=red] (v3) -- (v2); 
\draw [dashed] (v4) -- (v3) node [midway, fill=white, below=3pt] {$...$};
\draw (v5) -- (v4);
\draw (v6) -- (v5);

\draw (u1) -- (v1);
\draw[color=red] (u2) -- (v2);
\end{tikzpicture}} 
\hfil
  \resizebox{7cm}{!}{\begin{tikzpicture}[scale=1, Wvertex/.style={circle, draw=black, fill=white, scale=1}, bvertex/.style={circle, draw=black, fill=black, scale=0.3}]

\node [bvertex] (u0) at (-0.5, -0.5) {};
\node [bvertex] (u1) at (-3, 1) {};
\node [bvertex] (u2) at (-2, 1) {};
\node [bvertex] (u3) at (-1, 1) {};
\node [bvertex] (u4) at (0, 1) {};
\node [bvertex] (u5) at (1, 1) {};
\node [bvertex] (u6) at (2, 1) {};

\node [bvertex] (v0) at (3.5, 2.5) {};
\node [bvertex] (v1) at (1, 2) {};
\node [bvertex] (v2) at (2, 2) {};
\node [bvertex] (v3) at (3, 2) {};
\node [bvertex] (v4) at (4, 2) {};
\node [bvertex] (v5) at (5, 2) {};
\node [bvertex] (v6) at (6, 2) {};

\draw (u0) -- (u1);
\draw (u0) -- (u2);
\draw (u0) -- (u3);
\draw (u0) -- (u4);
\draw (u0) -- (u5);
\draw (u0) -- (u6);

\draw (u1) -- (u2);
\draw (u3) -- (u2); 
\draw [dashed] (u4) -- (u3) node [midway, above=3pt] {$...$};
\draw (u5) -- (u4);
\draw (u6) -- (u5);

\draw (v0) -- (v1);
\draw (v0) -- (v2);
\draw (v0) -- (v3);
\draw (v0) -- (v4);
\draw (v0) -- (v5);
\draw (v0) -- (v6);

\draw (v1) -- (v2);
\draw (v3) -- (v2); 
\draw [dashed] (v4) -- (v3) node [midway, below=3pt] {$...$};
\draw (v5) -- (v4);
\draw (v6) -- (v5);

\draw (u6) -- (v2);
\draw (u5) -- (v1);
\end{tikzpicture}}
    \caption{The graphs $G'_0$ (on the left) and $G_0$ (on the right).} 
    \label{fig:G_0'andG_0}
\end{figure}
\ref{fig:G_0'andG_0}.
We have
\[a_3(G_0')=a_3(G_0)-1=8q-33.\]
For any  $\lambda\in I$, we have
\begin{align*}
    g(\lambda) &= a_0 + \frac{a_1}{\lambda} + \frac{a_2}{\lambda^2} + \frac{a_3}{\lambda^3} + \frac{16q+O(1)}{\lambda^4} + \frac{32q+O(1)}{\lambda^5} + 
    O\left(\frac{q}{\lambda^6}\right)\\
    &\leq q-1 + \frac{2q-6}{\lambda} + \frac{4q-14}{\lambda^2} + \frac{8q-33}{\lambda^3} + 
    \frac{16q+O(1)}{\lambda^4} + \frac{32q+O(1)}{\lambda^5} + 
    O\left(\frac{q}{\lambda^6}\right)\\
    &< (q-1) + \frac{2q-6}{\lambda} + \frac{4q-14}{\lambda^2} + \frac{8q-32}{\lambda^3} 
    + \frac{16q-70}{\lambda^4} + \frac{32q-152}{\lambda^5} + 
    O\left(\frac{q}{\lambda^6}\right)\\
    &=f(\lambda).
\end{align*}
By Lemma \ref{lem:compare}, we have $\lambda_{2,max}=\lambda_g<\lambda_f=\lambda_2(G_0),$ a contradiction. The proof of the even case is finished.\vspace{0.5cm}

\noindent
\textit{\textbf{Case $n=2q+1$:}} We have
\[f(\lambda)=(q-1) + \frac{2q-4}{\lambda} + \frac{4q-12}{\lambda^2} + \frac{8q-32}{\lambda^3} + \frac{16q-64}{\lambda^4}+\frac{32q-112}{\lambda^5} + O\left(\frac{q}{\lambda^6}\right).\]

\noindent{\bf Claim 4:}  One of the following four cases must occur:
\begin{enumerate}
    \item $d_{u_1}=d_{u_2}=q-1$, and $N(u_1)\cap N(u_2)=\emptyset$.
    \item $d_{u_1}=d_{u_2}=q$, and $|N(u_1)\cap N(u_2)|=1$.
    \item $d_{u_1}=q$, $d_{u_2}=q-1$, and $N(u_1)\cap N(u_2)=\emptyset$.
    \item $d_{u_1}=q-1$, $d_{u_2}=q$, and $N(u_1)\cap N(u_2)=\emptyset$.
\end{enumerate}

The proof is identical to the proof of Claim 1 in the proof of Theorem \ref{thm:t3}.
Again, we can show that there is a cut vertex in both Case 1 and Case 2, which cannot occur since $G$ is 2-connected.
Cases 3 and 4 are symmetric. From now on, we assume
$d_{u_1}=q$, $d_{u_2}=q-1$, and $N(u_1)\cap N(u_2)=\emptyset$.\vspace{0.5cm}

\noindent{\bf Claim 5}: $G[{N(u_2)}]$ forms a path of length $q-2$.\vspace{0.5cm}

Otherwise, we have
\[ \sum_{(u,v) \in N(u_2)^2} w_1(u,v) =2 |E(G[{N(u_2)}])|\leq 2(q-3). \]
Let $g_2(\lambda)=\sum_{i=0}^\infty \lambda_2^{-i}\gamma_2'A_P^i\beta$ as in equation \eqref{eq:gamma2}.
Since $\frac{x_{u_1}}{x_{u_2}}<0$, 
using equation \eqref{eq:gamma2exp}, 
we have
\begin{align*}
    g_2(\lambda)
    &\leq q-1 + \frac{2q-6}{\lambda} + \frac{4q+O(1)}{\lambda^2} + \frac{8q+O(1)}{\lambda^3} + O\left(\frac{q}{\lambda^4}\right)\\
    &< (q-1) + \frac{2q-4}{\lambda} + \frac{4q-12}{\lambda^2} + \frac{8q-32}{\lambda^3} + O\left(\frac{q}{\lambda^4}\right)\\
    &=f(\lambda).
\end{align*}
By Lemma \ref{lem:compare}, we have $\lambda_{2,max}=\lambda_g<\lambda_f=\lambda_2(G_0),$ a contradiction and Claim 5 is proven.\vspace{0.5cm}

\noindent{\bf Claim 6}: $E(N(u_1), N(u_2))$ consists of two parallel edges.\vspace{0.5cm}

By Claim 5, the induced graph on $N(u_2)$ is a path $P_{q-1}$. Let $c= |E(N(u_1), N(u_2))|$. 
Since $G$ is 2-connected outerplanar graph,
we have $c=2$ or $c = 3$. When $c=2$, the two crossing edges form a matching.
When $c=3$, the three crossing edges must form a shape of $N$.

Let us estimate the contribution of these crossing edges in $\sum_{u,v\in N(u_2)} w_2(u,v)$.
Let $\eta$ be the contribution of crossing edges in $\sum_{u,v\in N(u_2)} w_2(u,v)$. We have
\[\eta=\begin{cases}
    2 & \mbox { if } c=2,\\
    5 & \mbox { if } c=3.
\end{cases}\]
Note the contribution of $P_{q-1}$ to $\sum_{u,v\in N(u_2)} w_2(u,v)$ is
$$2(q-1)-2 + 2(q-3)=4q-10.$$
For any  $\lambda\in I$, using \eqref{eq:gamma2} and \eqref{eqn:xu2xu1}, we have
\begin{align*}
    g_2(\lambda) &=\sum_{i=0}^\infty \lambda_2^{-i}\gamma_2^TA_P^i\beta\\
    &\leq q-1 + \frac{2q-4 + c \frac{x_{u_1}}{x_{u_2}}}{\lambda} + \frac{4q-10
    + \eta +O(1)\cdot \frac{x_{u_1}}{x_{u_2}} }{\lambda^2} + \frac{8q+O(1)}{\lambda^3} + \frac{16q+O(1)}{\lambda^4} +O\left(\frac{q}{\lambda^5}\right)\\
     &\leq q-1 + \frac{2q-4 - \frac{c^2}{\lambda}}{\lambda} + \frac{4q-10
    + \eta} {\lambda^2} + \frac{8q+O(1)}{\lambda^3} + \frac{16q+O(1)}{\lambda^4} + O\left(\frac{q}{\lambda^5}\right)\\
    &= (q-1) + \frac{2q-4}{\lambda} + \frac{4q-10 -c^2+\eta}{\lambda^2} + \frac{8q+O(1)}{\lambda^3} + \frac{16q+O(1)}{\lambda^4} + O\left(\frac{q}{\lambda^5}\right)\\
     &< (q-1) + \frac{2q-4}{\lambda} + \frac{4q-12}{\lambda^2} + \frac{8q-32}{\lambda^3} + \frac{16q-64}{\lambda^4} + O\left(\frac{q}{\lambda^5}\right)\\
    &=f(\lambda).
\end{align*}

If Claim 6 fails, we have $\eta-c^2=5-3^2=-4$. Thus, the last inequality holds.
By Lemma \ref{lem:compare}, we have $\lambda_{2,max}=\lambda_g<\lambda_f=\lambda_2(G_0),$ a contradiction and Claim 6 is proven.\vspace{0.5cm}

\noindent{\bf Claim 7}: $G[{N(u_1)}]$ forms a path of length $q-1$. The ending vertices of two edges in $E(N(u_1), N(u_2))$ have degrees
$1,1,2,2$ in $P$.\vspace{0.5cm}

Let $|E(G[{N(u_1)}])|=q-1-c'$. Since $G$ is 2-connected and outerplanar, we must have $c'=0$ or $1$. Write the two edges in $E(N(u_1), N(u_2))$
as $v_1w_1$ and $v_2w_2$ with $w_1,w_2\in N(u_2)$. Let $\eta(v_1,v_2)=1$ if $v_1v_2$ is not an edge, and $0$ otherwise. We have
\begin{align*}
    F_1(\lambda)&= q + \frac{2q-2-2c'}{\lambda} + \frac{4q+O(1)}{\lambda^2} +
    \frac{8q+O(1)}{\lambda^3}
    + O\left(\frac{q}{\lambda^4}\right)\\
     F_2(\lambda)&= q -1 + \frac{2q-4}{\lambda} + \frac{4q-8}{\lambda^2} +
    \frac{8q-10 -2(d^P(w_1)+d^P(w_2))-2\eta(v_1,v_2)}{\lambda^3} 
    + \frac{16q+O(1)}{\lambda^4}+\frac{32q+O(1)}{\lambda^5}  \\
    &+ O\left(\frac{q}{\lambda^6}\right)\\
    D(\lambda) &= \frac{2}{\lambda} +\frac{d^P(v_1)+d^P(w_1)+d^P(v_2)+d^P(w_2)
    }{\lambda^2}
\end{align*}

Therefore, 
\begin{align*}
    \lambda^2 &=\frac{1}{2}
\left( (F_1(\lambda) +F_2(\lambda)) 
- \sqrt{ (F_1(\lambda) -F_2(\lambda))^2 + 4 D(\lambda)^2}
\right)\\
&=\frac{1}{2}
\left( (F_1(\lambda) +F_2(\lambda)) 
- (F_1(\lambda)-F_2(\lambda) \sqrt{ 1+ 
\frac{4 D(\lambda)^2}{(F_1(\lambda) -F_2(\lambda))^2} }
\right)\\
&=\frac{1}{2}
\left( (F_1(\lambda) +F_2(\lambda)) 
- (F_1(\lambda)-F_2(\lambda) \left( 1+ 
\frac{2 D(\lambda)^2}{(F_1(\lambda) -F_2(\lambda))^2} + O\left( \frac{1}{(q-1)^2}\right)
\right)
\right)\\
&= F_2(\lambda)- \frac{D(\lambda)^2}{(F_1(\lambda) -F_2(\lambda))}+ O\left( \frac{1}{(q-1)^2}\right)\\
&= q -1 + \frac{2q-4}{\lambda} + \frac{4q-8}{\lambda^2} +
    \frac{8q-10- 2(d^P(w_1)+d^P(w_2))-2\eta(v_1,v_2)}{\lambda^3} 
    + \frac{16q+O(1)}{\lambda^4}+\frac{32q+O(1)}{\lambda^5} +O\left(\frac{q}{\lambda^6}\right) \\
 &\hspace*{2mm}   
    - \left(\frac{2}{\lambda} +\frac{d^P(v_1)+d^P(w_1)+d^P(v_2)+d^P(w_2)
    }{\lambda^2}\right)^2 \left(1-\frac{2-2c'}{\lambda}\right)   
    + O\left(\frac{1}{(q-1)^2}\right)\\
&= (q-1) + \frac{2q-4}{\lambda} + \frac{4q-12}{\lambda^2} + \frac{8q-2 - (4d^P(v_1)+6d^P(w_1)+4d^P(v_2)+6d^P(w_2))-2\eta(v_1,v_2)-8c'
}{\lambda^3}  \\
 &\hspace*{2mm} 
+ \frac{16q+O(1)}{\lambda^4}+\frac{32q+O(1)}{\lambda^5} +O\left(\frac{q}{\lambda^6}\right).
\end{align*}
Since $G[{N(u_2)}]=P_{q-1}$, we must have
$$d^P(w_1) + d^P(w_2)\geq 1+2=3.$$
When $c'=0$, we have
$$ 4(d^P(v_1)+d^P(v_2))+6(d^P(w_1)+d^P(w_2))\geq 4(1+2)+ 6(1+2)=30,$$
with equality if and only if
the ending vertices of two edges in $E(N(u_1), N(u_2))$ have degrees
$1,1,2,2$ in $P$.
When $c'=1$, $v_1v_2$ cannot be an edge, as otherwise $u_1$ would be a cut vertex in $G$. We have
$$ 4(d^P(v_1)+d^P(v_2))+6(d^P(w_1)+d^P(w_2))+2\eta(v_1,v_2)\geq 4(1+0)+ 6(1+2)+2=24.$$
Therefore, if Claim 7 fails, then
we have
\[ 8q-8 - (6d^P(v_1)+4d^P(w_1)+6d^P(v_2)+4d^P(w_2))-8c'\leq 8q-33.\]
And so,
\begin{align*}
g(\lambda)
&\leq
(q-1) + \frac{2q-4}{\lambda} + \frac{4q-12}{\lambda^2} + \frac{8q-33
}{\lambda^3}  + \frac{16q+O(1)}{\lambda^4}+\frac{32q+O(1)}{\lambda^5} +O\left(\frac{q}{\lambda^6}\right)\\
&<
 (q-1) + \frac{2q-4}{\lambda} + \frac{4q-12}{\lambda^2} + \frac{8q-32}{\lambda^3} + \frac{16q-64}{\lambda^4}+\frac{32q-112}{\lambda^5} +O\left(\frac{q}{\lambda^6}\right)\\
 &=f(\lambda).
\end{align*}
By Lemma \ref{lem:compare}, we have $\lambda_{2,max}=\lambda_g<\lambda_f=\lambda_2(G_0),$ a contradiction and Claim 7 is proven.
Thus, $G$ must be one of the following graphs.
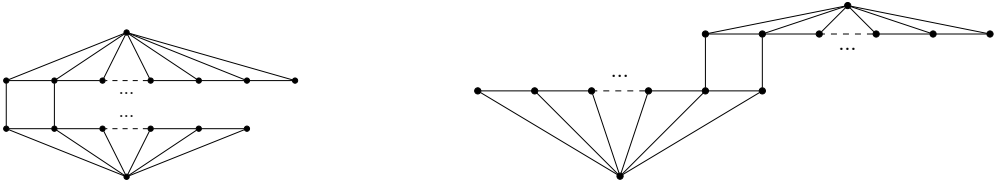
\begin{figure}[hbt]
    \centering
    \resizebox{4cm}{!}{\begin{tikzpicture}[scale=1, Wvertex/.style={circle, draw=black, fill=white, scale=1}, bvertex/.style={circle, draw=black, fill=black, scale=0.3}]

\node [bvertex] (u0) at (-0.5, -1.5) {};
\node [bvertex] (u1) at (-3, -0.5) {};
\node [bvertex] (u2) at (-2, -0.5) {};
\node [bvertex] (u3) at (-1, -0.5) {};
\node [bvertex] (u4) at (0, -0.5) {};
\node [bvertex] (u5) at (1, -0.5) {};
\node [bvertex] (u6) at (2, -0.5) {};

\node [bvertex] (v0) at (-0.5, 1.5) {};
\node [bvertex] (v1) at (-3, 0.5) {};
\node [bvertex] (v2) at (-2, 0.5) {};
\node [bvertex] (v3) at (-1, 0.5) {};
\node [bvertex] (v4) at (0, 0.5) {};
\node [bvertex] (v5) at (1, 0.5) {};
\node [bvertex] (v6) at (2, 0.5) {};
\node [bvertex] (v7) at (3, 0.5) {};

\draw (u0) -- (u1);
\draw (u0) -- (u2);
\draw (u0) -- (u3);
\draw (u0) -- (u4);
\draw (u0) -- (u5);
\draw (u0) -- (u6);
\draw (v0) -- (v7);

\draw (u1) -- (u2);
\draw (u3) -- (u2); 
\draw [dashed] (u4) -- (u3) node [midway, above=3pt] {$...$};
\draw (u5) -- (u4);
\draw (u6) -- (u5);
\draw (v6) -- (v7);

\draw (v0) -- (v1);
\draw (v0) -- (v2);
\draw (v0) -- (v3);
\draw (v0) -- (v4);
\draw (v0) -- (v5);
\draw (v0) -- (v6);

\draw (v1) -- (v2);
\draw (v3) -- (v2); 
\draw [dashed] (v4) -- (v3) node [midway, below=3pt] {$...$};
\draw (v5) -- (v4);
\draw (v6) -- (v5);

\draw (u1) -- (v1);
\draw (u2) -- (v2);
\end{tikzpicture}}
    \hfil
  \resizebox{7cm}{!}{\begin{tikzpicture}[scale=1, Wvertex/.style={circle, draw=black, fill=white, scale=1}, bvertex/.style={circle, draw=black, fill=black, scale=0.3}]

\node [bvertex] (u0) at (-0.5, -0.5) {};
\node [bvertex] (u1) at (-3, 1) {};
\node [bvertex] (u2) at (-2, 1) {};
\node [bvertex] (u3) at (-1, 1) {};
\node [bvertex] (u4) at (0, 1) {};
\node [bvertex] (u5) at (1, 1) {};
\node [bvertex] (u6) at (2, 1) {};

\node [bvertex] (v0) at (3.5, 2.5) {};
\node [bvertex] (v1) at (1, 2) {};
\node [bvertex] (v2) at (2, 2) {};
\node [bvertex] (v3) at (3, 2) {};
\node [bvertex] (v4) at (4, 2) {};
\node [bvertex] (v5) at (5, 2) {};
\node [bvertex] (v6) at (6, 2) {};

\draw (u0) -- (u1);
\draw (u0) -- (u2);
\draw (u0) -- (u3);
\draw (u0) -- (u4);
\draw (u0) -- (u5);
\draw (u0) -- (u6);

\draw (u1) -- (u2);
\draw (u3) -- (u2); 
\draw [dashed] (u4) -- (u3) node [midway, above=3pt] {$...$};
\draw (u5) -- (u4);
\draw (u6) -- (u5);

\draw (v0) -- (v1);
\draw (v0) -- (v2);
\draw (v0) -- (v3);
\draw (v0) -- (v4);
\draw (v0) -- (v5);
\draw (v0) -- (v6);

\draw (v1) -- (v2);
\draw (v3) -- (v2); 
\draw [dashed] (v4) -- (v3) node [midway, below=3pt] {$...$};
\draw (v5) -- (v4);
\draw (v6) -- (v5);

\draw (u6) -- (v2);
\draw (u5) -- (v1);
\end{tikzpicture}}
\caption{The graphs $G'_0$ (on the left) and $G_0$ (on the right).} 
    \label{fig:enter-label}
\end{figure}\\
It suffices to show $G$ is not $G'_0$.
In fact, for $G'_0$, we can calculate $F_1(\lambda)$, $F_2(\lambda)$, and $D(\lambda)$
as follows:
\begin{align*}
F_1(\lambda)&=  q + \frac{2q-2}{\lambda} + \frac{4q-4}{\lambda^2} + \frac{8q-8}{\lambda^3} + \frac{16q-12}{\lambda^4}+\frac{32q-14}{\lambda^5} + \frac{64q+5}{\lambda^6} 
+O\left(\frac{q}{\lambda^7}\right),\\
       F_2(\lambda) &=  (q-1) + \frac{2q-4}{\lambda} + \frac{4q-8}{\lambda^2} + \frac{8q-16}{\lambda^3} + \frac{16q-28}{\lambda^4}+\frac{32q-46}{\lambda^5} 
       + \frac{64q-59}{\lambda^6}+O\left(\frac{q}{\lambda^7}\right).\\
    D(\lambda)&=  \frac{2}{\lambda} + \frac{6}{\lambda^2} + \frac{17}{\lambda^3} + \frac{44}{\lambda^4}+\frac{112}{\lambda^5}
    +\frac{276}{\lambda^6}
    +O\left(\frac{q}{\lambda^7}\right).
\end{align*}

Therefore,
\begin{align*}
    g(\lambda) &=\frac{1}{2}
\left( (F_1(\lambda) +F_2(\lambda)) 
- \sqrt{ (F_1(\lambda) -F_2(\lambda))^2 + 4 D(\lambda)^2}
\right)\\
&= \frac{1}{2}
\left(2q-1 + \frac{4q-6}{\lambda} + \frac{8q-12}{\lambda^2} + \frac{16q-24}{\lambda^3} + \frac{32q-40}{\lambda^4}+\frac{64q-60}{\lambda^5} +\frac{128q-54}{\lambda^6} 
 +O\left(\frac{q}{\lambda^7}\right)\right.\\
& 
\left.
-\sqrt{
\left(1+ \frac{2}{\lambda} + \frac{4}{\lambda^2} + \frac{8}{\lambda^3} + \frac{16}{\lambda^4}+\frac{32}{\lambda^5} +\frac{64}{\lambda^6}+O\left(\frac{q}{\lambda^7}\right)\right)^2
+ 4\left(\frac{2}{\lambda} + \frac{6}{\lambda^2} + \frac{17}{\lambda^3} + \frac{44}{\lambda^4}+\frac{112}{\lambda^5} +\frac{276}{\lambda^6} +O\left(\frac{q}{\lambda^7}\right)\right)^2
}
\right) \\
&= (q-1) + \frac{2q-4}{\lambda} + \frac{4q-12}{\lambda^2} + \frac{8q-32}{\lambda^3} + \frac{16q-68}{\lambda^4}+\frac{32q-122}{\lambda^5} +\frac{64q-244}{\lambda^6} +O\left(\frac{q}{\lambda^7}\right)\\
&<(q-1) + \frac{2q-4}{\lambda} + \frac{4q-12}{\lambda^2} + \frac{8q-32}{\lambda^3} + \frac{16q-64}{\lambda^4}+\frac{32q-112}{\lambda^5} +\frac{64q-232}{\lambda^6} +O\left(\frac{q}{\lambda^7}\right)\\
&=f(\lambda).
\end{align*}

Thus, we have $\lambda_2(G_0')<\lambda_2(G_0)$. The crucial difference between $G_0'$ and $G_0$ is that there are $17$ $(u, v)$-walks of length $3$ in $G_0'$ for $q\ge 5$ with $(u,v) \in N(u_1) \times N(u_2)$, but only $16$ such walks in $G_0$ (as can be seen by comparing the coefficients of $\lambda^{-3}$ in $D(\lambda)$ for $G_0$ and $G_0'$). 
This completes the proof that $G_0$ is the unique extemal graph among 
all 2-connected outerplanar graphs on $n$ vertices.
\end{proof}

\section{Future directions}

In this paper, we have determined the leading order asymptotics of $\lambda_{k, max}$ for any fixed $k$ and also determined the precise extremal graphs for $k=2$ in the class of $n$-vertex outerplanar graphs, connected $n$-vertex outerplanar graphs, and $2$-connected $n$-vertex outerplanar graphs. We make the following conjectures for the extremal graphs for $k=3$. 

\begin{conjecture}\label{conj3q}
    For $n$ sufficiently large, among all connected outerplanar graphs on $n=3q$ vertices, the graph maximizing $\lambda_3$ is unique and isomorphic to the following graph, denoted by $(K_1\vee P_{q-1})\!\!-\!\!(K_1\vee P_{q-1})\!\!-\!\!(K_1\vee P_{q-1})$, which is constructed by gluing three disjoint copies of the fan graph $K_1\vee P_{q-1}$ via edges connecting vertices of smallest degrees to largest degree between all three components (see Figure~\ref{fig:outerplanar_3q}).
\end{conjecture}

\begin{conjecture}\label{conj3q+1}
    For $n$ sufficiently large, among all connected outerplanar graphs on $n=3q+1$ vertices, the graph maximizing $\lambda_3$ is unique and isomorphic to the following graph, denoted by $3(K_1\vee P_{q-1})\!\!-\!\!K_1$, which is constructed by gluing three disjoint copies of the fan graph $K_1\vee P_{q-1}$ to a $K_1$ via edges connecting any vertices (but maintaining outerplanarity) to that $K_1$.
\end{conjecture}

\begin{figure}
    \begin{center}
     \resizebox{7.5cm}{!}{\begin{tikzpicture}[scale=1, Wvertex/.style={circle, draw=black, fill=white, scale=1}, bvertex/.style={circle, draw=black, fill=black, scale=0.3}]
		
    \node [bvertex] (u0) at (2, -1) {};
    \node [bvertex] (u1) at (0, -3) {};
    \node [bvertex] (u2) at (-0.5, -2.5) {};
    \node [bvertex] (u3) at (-1.5, -1.5) {};
    \node [bvertex] (u4) at (-2, -1) {};

    \node [bvertex] (v0) at (2, 3) {};
    \node [bvertex] (v1) at (5, 3) {};
    \node [bvertex] (v2) at (4.6, 2.5) {};
    \node [bvertex] (v3) at (3.85, 1.5) {};
    \node [bvertex] (v4) at (3.5, 1) {};
    
    \node [bvertex] (w0) at (-3.5, 1) {};
    \node [bvertex] (w1) at (-5, 3) {};
    \node [bvertex] (w2) at (-4.5, 3) {};
    \node [bvertex] (w3) at (-2.5, 3) {};
    \node [bvertex] (w4) at (-2, 3) {};

    \draw (u0) -- (u1);
    \draw (u0) -- (u2);
    \draw (u0) -- (u3);
    \draw (u0) -- (u4);
    
    \draw (u1) -- (u2);
    \draw (u3) -- (u4); 
    \draw [dashed] (u2) -- (u3) node [midway, rotate=135, above=5pt]{}; 

    \draw (v0) -- (v1);
    \draw (v0) -- (v2);
    \draw (v0) -- (v3);
    \draw (v0) -- (v4);
    
    \draw (v1) -- (v2);
    \draw (v3) -- (v4); 
    \draw [dashed] (v2) -- (v3) node [midway, rotate=-125, above=5pt]{}; 

    \draw (w0) -- (w1);
    \draw (w0) -- (w2);
    \draw (w0) -- (w3);
    \draw (w0) -- (w4);
    
    \draw (w1) -- (w2);
    \draw (w3) -- (w4); 
    \draw [dashed] (w2) -- (w3) node [midway, above=5pt]{};
    
    \draw (w0) -- (u4);
    \draw (u0) -- (v4);
    \draw (v0) -- (w4);
    
\end{tikzpicture}}
    \end{center}
    \caption{The conjectured outerplanar graph maximizing $\lambda_3$ on $n=3q$ vertices.}
    \label{fig:outerplanar_3q}
    \end{figure}

Note that Conjecture~\ref{conj3q+1} is a special case of Conjecture~\ref{conjkq+1}. 

\begin{conjecture}\label{cong3q+2}
    For $n$ sufficiently large, among all connected outerplanar graphs on $n=3q+2$ vertices, the graph maximizing $\lambda_3$ is unique and isomorphic to the graph obtained after deleting a lowest degree vertex from the connected outerplanar graph that (we conjecture) maximizes $\lambda_3$ on $3(q+1)$ vertices
\end{conjecture}

We have tested these conjectures for relatively small values of $n$ by first narrowing down the possible structures of the graph maximizing $\lambda_3$ over all connected outerplanar graphs on $n$ vertices. 

We can ask the following general question. Given a family $\cal F$ of graphs on $n$ vertices, which graph in $\cal F$ has the maximum $\lambda_k$?  In particular, we are interested in planar graphs and $K_{s,t}$-minor free graphs. 

\section*{Acknowledgement}
We thank the anonymous referee for their thorough and detailed review of this paper, which has significantly improved the quality of the manuscript.

\end{document}